\documentclass[12pt]{amsart}

\usepackage{amsmath}
\usepackage{amssymb}
\usepackage{amsthm}
\usepackage{pstricks}
\usepackage{pst-node}
\usepackage{algorithmic}
\usepackage{algorithm}
\usepackage{cite}
\algsetup{indent=2em}

\usepackage{verbatim}

\numberwithin{equation}{section}

\def\wt{\operatorname{wt}}
\newcommand{\procname}[1]{\texttt{#1}}
\def\RandomPoint{\procname{RandomPoint}}
\def\RandomPointOn{\procname{RandomPointOn}}
\def\RandomLine{\procname{RandomLine}}
\def\RandomLineThrough{\procname{RandomLineThrough}}
\def\DecorateSide{\procname{DecorateSide}}
\def\DecorateVertex{\procname{DecorateVertex}}
\def\ConstructSideB{\procname{ConstructSide2}}
\def\DecorateSideB{\procname{DecorateSide2}}
\def\ConstructVertexB{\procname{ConstructVertex2}}
\def\DecorateVertexB{\procname{DecorateVertex2}}
\def\DecorateRandomly{\procname{DecorateRandomly}}
\def\main{\procname{main}}
\def\TripleConjugate{\procname{TripleConjugate}}
\def\ProjectiveTransformation{\procname{ProjectiveTransformation}}
\def\ProjectiveTransformationB{\procname{ProjectiveTransformation2}}

\newcommand{\meet}[2]{#1 \cap #2}
\newcommand{\join}[2]{\overleftrightarrow{#1 #2}}

\newtheorem{thm}{Theorem}[section]
\newtheorem{prop}[thm]{Proposition}
\newtheorem{lem}[thm]{Lemma}
\newtheorem{cor}[thm]{Corollary}

\theoremstyle{definition}
\newtheorem{rem}[thm]{Remark}
\newtheorem{defin}[thm]{Definition}

\newtheorem{conj}[thm]{Conjecture}

\begin{document}
\title[On singularity confinement for the pentagram map]{On singularity confinement \\ for the pentagram map}
\author{Max Glick}
\address{Department of Mathematics, University of Michigan,
Ann Arbor, MI 48109, USA} \email{maxglick@umich.edu}


\begin{abstract}
The pentagram map, introduced by R. Schwartz, is a birational map on the configuration space of polygons in the projective plane.

We study the singularities of the iterates of the pentagram map. We show that a ``typical'' singularity disappears after a finite number of iterations, a confinement phenomenon first discovered by Schwartz.  We provide a method to bypass such a singular patch by directly constructing the first subsequent iterate that is well-defined on the singular locus under consideration. The key ingredient of this construction is the notion of a decorated (twisted) polygon, and the extension of the pentagram map to the corresponding decorated configuration space.
\end{abstract}

\date{\today}
\thanks{Partially supported by NSF grants DMS-0943832 and DMS-1101152.} 

\maketitle

\tableofcontents

\section{Introduction}
The pentagram map, introduced by R. Schwartz \cite{S2}, is a geometric construction which produces one polygon from another.  Successive applications of this operation (cf. Figure \ref{figureT}) define a discrete dynamical system that has received considerable attention in recent years (see, e.g., \cite{S,OST,G,So,OST2}) due to its integrability properties and its connections to moduli spaces and cluster algebras.  This paper is devoted to the study of singularity confinement for the pentagram map, a phenomenon first observed experimentally by Schwartz.  Informally speaking, a singularity of a map at a point is said to be confined if some higher iterate of the map is well-defined at that point.  We investigate singularities of the pentagram map and prove confinement in several cases.

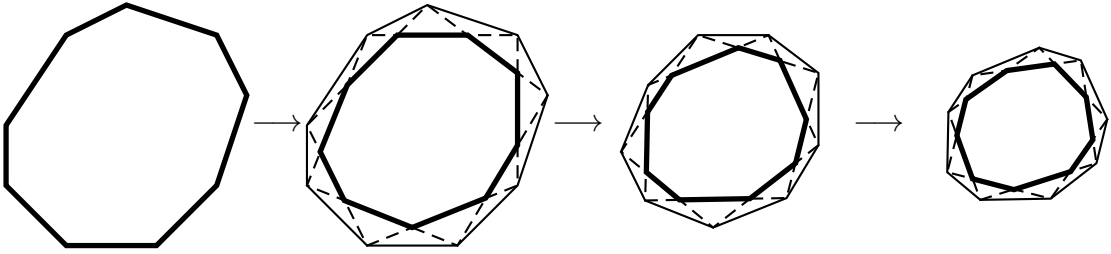
\begin{figure}[ht] \label{figureT}
\psset{unit=.8cm}
\begin{pspicture}(20,5)
\rput(-.5,-.5){
	\pspolygon[linewidth=2pt](1,2)(1,3)(2,4.5)(3,5)(4.5,4.5)(5,3.5)(4.5,2)(3.5,1)(2,1)
  \rput(5.5,3){$\longrightarrow$}

  \rput(5,0){
  \pspolygon(1,2)(1,3)(2,4.5)(3,5)(4.5,4.5)(5,3.5)(4.5,2)(3.5,1)(2,1)
  \pspolygon[linestyle=dashed](1,2)(2,4.5)(4.5,4.5)(4.5,2)(2,1)(1,3)(3,5)(5,3.5)(3.5,1)
  \pspolygon[linewidth=2pt](1.22,2.56)(1.67,3.67)(2.5,4.5)(3.67,4.5)(4.5,3.87)(4.5,2.67)(3.97,1.79)(2.75,1.3)(1.62,1.75)
  \rput(5.5,3){$\longrightarrow$}
  }
  
  \rput(10,0){
  \pspolygon(1.22,2.56)(1.67,3.67)(2.5,4.5)(3.67,4.5)(4.5,3.87)(4.5,2.67)(3.97,1.79)(2.75,1.3)(1.62,1.75)
  \pspolygon[linestyle=dashed](1.22,2.56)(2.5,4.5)(4.5,3.87)(3.97,1.79)(1.62,1.75)(1.67,3.67)(3.67,4.5)(4.5,2.67)(2.75,1.3)
  \pspolygon[linewidth=2pt](1.64,2.22)(1.66,3.22)(2.06,3.83)(3.17,4.29)(3.86,4.08)(4.30,3.10)(4.12,2.37)(3.36,1.78)(2.19,1.76)
  \rput(5.5,3){$\longrightarrow$}
  }
  
  \rput(15,0){
  \pspolygon(1.64,2.22)(1.66,3.22)(2.06,3.83)(3.17,4.29)(3.86,4.08)(4.30,3.10)(4.12,2.37)(3.36,1.78)(2.19,1.76)
  \pspolygon[linestyle=dashed](1.64,2.22)(2.06,3.83)(3.86,4.08)(4.12,2.37)(2.19,1.76)(1.66,3.22)(3.17,4.29)(4.30,3.1)(3.36,1.78)
  \pspolygon[linewidth=2pt](1.8,2.83)(1.95,3.43)(2.63,3.91)(3.42,4.02)(3.95,3.47)(4.06,2.76)(3.69,2.23)(2.75,1.93)(2.06,2.11)
  }
}
\end{pspicture}
\psset{unit=1cm}
\caption{Three iterations of the pentagram map in the space of 9-gons}
\end{figure}

The pentagram map is typically defined for objects called twisted polygons defined by Schwartz \cite{S}.  A \emph{twisted polygon} is a sequence $A=(A_i)_{i\in \mathbb{Z}}$ of points in the projective plane that is periodic modulo some projective transformation $\phi$, i.e., $A_{i+n} = \phi(A_i)$ for all $i \in \mathbb{Z}$.  We will place the additional restriction that every quadruple of consecutive points of $A$ be in general position.  Two twisted polygons $A$ and $B$ are said to be \emph{projectively equivalent} if there exists a projective transformation $\psi$ such that $\psi(A_i) = B_i$ for all $i$.  Let $\mathcal{P}_n$ denote the space of twisted $n$-gons modulo projective equivalence.  

It is convenient to also allow twisted polygons to be indexed by $\frac{1}{2}+\mathbb{Z}$ instead of $\mathbb{Z}$.  Let $\mathcal{P}_n^*$ denote the space of twisted $n$-gons indexed by $\frac{1}{2} + \mathbb{Z}$, modulo projective equivalence.

The \emph{pentagram map}, denoted $T$, inputs a twisted polygon $A$ and constructs a new twisted polygon $B$ defined by $B_i = \meet{\join{A_{i-\frac{3}{2}}}{A_{i+\frac{1}{2}}}}{\join{A_{i-\frac{1}{2}}}{A_{i+\frac{3}{2}}}}$.  Note that if $A$ is indexed by $\mathbb{Z}$ then $B$ is indexed by $\frac{1}{2}+\mathbb{Z}$ and vice versa.  The pentagram map preserves projective equivalence, so it induces maps
\begin{align*}
\alpha_1 &: \mathcal{P}_n^* \to \mathcal{P}_n \\
\alpha_2 &: \mathcal{P}_n \to \mathcal{P}_n^* 
\end{align*}

Schwartz \cite{S} gives coordinates $x_1,\ldots, x_{2n}$ defined generically on $\mathcal{P}_n$ and on $\mathcal{P}_n^*$.  These are naturally ordered cyclically, so let $x_{i+2n} = x_i$ for all $i \in \mathbb{Z}$.  Expressed in these coordinates, the maps $\alpha_1$ and $\alpha_2$ take a simple form.

\begin{prop}[{\cite[(7)]{S}}] \label{propalpha12x} Suppose that $(x_1,\ldots,x_{2n})$ are the $x$-coordinates of $A$.  If $A \in \mathcal{P}_n^*$ then 
\begin{equation}
x_j(\alpha_1(A)) = \begin{cases} 
                      x_{j-1}\frac{1-x_{j-3}x_{j-2}}{1-x_{j+1}x_{j+2}},
                      &j \textrm{ even} \\ 
                      x_{j+1}\frac{1-x_{j+3}x_{j+2}}{1-x_{j-1}x_{j-2}},
                      &j \textrm{ odd} \\ 
                   \end{cases} 
\label{alpha1x}
\end{equation}
Alternately, if $A \in \mathcal{P}_n$ then
\begin{equation}
x_j(\alpha_2(A)) = \begin{cases} 
                       x_{j+1}\frac{1-x_{j+3}x_{j+2}}{1-x_{j-1}x_{j-2}},
                       &j \textrm{ even} \\ 
                       x_{j-1}\frac{1-x_{j-3}x_{j-2}}{1-x_{j+1}x_{j+2}},
                       &j \textrm{ odd} \\ 
                   \end{cases}
\label{alpha2x}
\end{equation}
\end{prop}

We will be interested in $T^k$, the $k$th iterate of the pentagram map.  Defined on $\mathcal{P}_n$ it takes the form $T^k = \underbrace{\cdots \circ \alpha_2 \circ \alpha_1 \circ \alpha_2}_k$ and has image in either $\mathcal{P}_n$ or $\mathcal{P}_n^*$ depending on the parity of $k$.  By \eqref{alpha1x} and \eqref{alpha2x}, $T^k$ is a rational map.  The purpose of this paper is to better understand the singularities of the pentagram map and its iterates.

Let $A \in \mathcal{P}_n$ be a singular point of the pentagram map.  Then typically $A$ will be a singular point of $T^k$ for all $k$ less than some $m$, but not of $T^m$.  This phenomenon is known as \emph{singularity confinement} and was identified by Grammaticos, Ramani, and Papageorgiou \cite{GRP} as a feature common to many discrete integrable systems.  Now, the pentagram map is a discrete integrable system as proven by Ovsienko, Schwartz, and Tabachnikov \cite{OST,OST2} and Soloviev \cite{So}.  That singularity confinement holds in this setting has been observed experimentally by Schwartz.  The current paper seeks to understand singularity confinement for the pentagram map from both an algebraic and geometric perspective.

Algebraically, \eqref{alpha2x} suggests that a polygon $A \in \mathcal{P}_n$ is a singular point of the pentagram map whenever $x_{2i}(A)x_{2i+1}(A) = 1$ for some $i \in \mathbb{Z}$.  To check how many steps the singularity persists, one must determine for which $k$ the rational expression for $T^k$ has a vanishing denominator at the given point.  We use generating function formulas for these denominators from \cite{G} to better understand when this occurs.

What we discover is that the behavior of a singularity seems to depend on the set $S$ of integers $i$ for which $x_{2i}(A)x_{2i+1}(A) = 1$.  We call $S$ the \emph{type} of the singularity and attempt to understand when singularity confinement holds for generic polygons of a given type.  The simplest case is when $S$ consists of a single element, in which event the singularity is confined to two iterations (i.e. $A$ is a singular point of $T$ and $T^2$ but not of $T^3$).  More generally, suppose $S$ is a finite arithmetic progression with common difference equal to 1 or 2.  We prove that generic singularities of these types are confined to $l+1$ steps where $l$ is the size of the arithmetic progression.  

We do not have as complete an understanding of the situation for other singularity types.  If the number of sides $n$ of the polygon is odd, we show that singularity confinement holds generically for every type except the worst case $S=\{1,\ldots,n\}$.  In addition we have an upper bound for the number of iterations such singularities last.  The case of $n$ even seems to be more complicated and we only have a conjectural answer as to which types exhibit singularity confinement.

From a geometric perspective, the condition $x_{2i}(A)x_{2i+1}(A) = 1$ indicates that the triple of vertices $A_{i-2}$, $A_i$, and $A_{i+2}$ are collinear.  Although one can construct $B=T(A)$ in this case, the result will violate the condition that quadruples of consecutive vertices be in general position.  In fact, $B_{i-\frac{3}{2}}$, $B_{i-\frac{1}{2}}$, $B_{i+\frac{1}{2}}$, and $B_{i+\frac{3}{2}}$ will be collinear making it impossible to carry the construction any further.  The notion of singularity confinement also has a geometric interpretation.  If $A$ has a singularity which vanishes after $m$ steps, then one can approximate $A$ by nonsingular polygons, apply the construction $T^m$ to them, and take a limit to find $T^m(A)$.  Since $A$ is a regular point of $T^m$, the result of this procedure does not depend on the approximations of $A$.

Our main result on the geometric side is a straightedge construction of the first defined iterate $T^m(A)$ of a polygon $A$ of certain singularity types.  The basic idea is to fix, up to the first order, a family of approximations of $A$ by nonsingular polygons.  The data needed to accomplish this is encoded by a collection of points and lines which we call a decoration of $A$.  With this done, the iterates between $A$ and $T^m(A)$ become well-defined.  To determine $T^m(A)$, we iterate a procedure which constructs these intermediate polygons one by one.

This paper is organized as follows.  Section 2 reviews previous work on the pentagram map, including a non-recursive formula for $T^k$ as a rational map of the $x$-coordinates.  This map factors into polynomials, some properties of which are given in Section 3.  Section 4 identifies a hierarchy of singularity types of the pentagram map and establishes that generic polygons of these types exhibit singularity confinement.  The remainder of the paper addresses the problem of moving past singularities by constructing $T^m(A)$ from $A$ when $A$ is a singular point of $T,T^2,\ldots,T^{m-1}$.  An approach which works for the simplest singularity type is given in Section 5 along with a discussion of its limitations in handling more severe singularities.  Section~6 introduces decorated polygons which will serve as the underlying objects of the main construction.  In Section 7 we develop the procedure which is iterated in our main construction.  Section 8 states the main construction itself and discusses what is needed to prove its correctness for a given singularity type.  All algorithms are stated explicitly, but some contain steps which are nontrivial to accomplish via a straightedge construction.  In Appendix A we fill in the details for these steps.

The following notation will be used throughout.  If $a,b,k \in \mathbb{Z}$, $a \leq b$, $k \geq 1$ and $a \equiv b \pmod{k}$ then let $[a,b]_k$ denote the arithmetic progression
\begin{displaymath}
[a,b]_k = \{a,a+k,a+2k,\ldots,b\}
\end{displaymath}
Twisted polygons will be denoted by capital letters with individual vertices indexed by either $\mathbb{Z}$ or $\frac{1}{2} + \mathbb{Z}$.  The sides of a polygon (i.e. lines passing through two consecutive vertices) will be denoted by the corresponding lowercase letter and indexed using the opposite indexing scheme.  For instance, if $A$ is a twisted polygon indexed by $\mathbb{Z}$ then its vertices are denoted $A_i$ for $i \in \mathbb{Z}$ and its sides are denoted $a_j = \join{A_{j-\frac{1}{2}}}{A_{j+\frac{1}{2}}}$ for $j \in (\frac{1}{2} + \mathbb{Z})$.

\medskip

\textbf{Acknowledgments.} I thank Pavlo Pylyavskyy, Sergei Tabachnikov, and particularly Sergey Fomin for many valuable discussions during the course of this project.  I am grateful to Richard Schwartz who introduced me to this problem and explained to me his previous work in the area.

\section{Pentagram map background}
The \emph{cross ratio} of four real numbers $a,b,c,d$ is defined to be 
\begin{displaymath}
[a,b,c,d] = \frac{(a-b)(c-d)}{(b-c)(d-a)}
\end{displaymath}
This definition extends to the projective line, on which it gives a projective invariant of four points.  We will be interested in taking the cross ratio of four collinear points in the projective plane, or dually, the cross ratio of four lines intersecting at a common point.  To be consistent with notation from \cite{G}, let $\chi(a,b,c,d)$ denote the cross ratio taken in a different order: $\chi(a,b,c,d) = [b,a,c,d]$.

Let $A$ be a twisted polygon.  The $x$-coordinates of $A$ are defined by Schwartz \cite{S} as follows.  For each index $k$ of $A$ let
\begin{align*}
x_{2k}(A) &= \chi (A_{k-2},A_{k-1},B,D) \\
x_{2k+1}(A) &= \chi(A_{k+2},A_{k+1},C,D). 
\end{align*}
where $B=\meet{\join{A_{k-2}}{A_{k-1}}}{\join{A_k}{A_{k+1}}}$, $C=\meet{\join{A_{k-1}}{A_k}}{\join{A_{k+1}}{A_{k+2}}}$, and $D=\meet{\join{A_{k-2}}{A_{k-1}}}{\join{A_{k+1}}{A_{k+2}}}$ (see Figure \ref{figDefx}).  Now $x_{j+2n}=x_j$ for all $j \in \mathbb{Z}$, and as mentioned in the introduction, $x_1,\ldots,x_{2n}$ give a set of coordinates on $\mathcal{P}_n$ and on $\mathcal{P}_n^*$.

\begin{figure}
\begin{pspicture}(10,6)
\rput(2,0){
\psline[showpoints=true,linewidth=2pt](1,1)(3,1)(5,2)(5,3)(4,4)
\uput[d](1,1){$A_{k-2}$}
\uput[d](3,1){$A_{k-1}$}
\uput[dr](5,2){$A_k$}
\uput[ur](5,3){$A_{k+1}$}
\uput[ur](4,4){$A_{k+2}$}
\psline[showpoints=true,dotsize=2pt 5](3,1)(5,1)(5,2)(5.67,2.33)(5,3)
\psline[showpoints=true,dotsize=2pt 5](5,1)(7,1)(5.67,2.33)
\uput[d](5,1){$B$}
\uput[ur](5.67,2.33){$C$}
\uput[dr](7,1){$D$}
}
\end{pspicture}
\caption{The points involved in the definitions of $x_{2k}(A)$ and $x_{2k+1}(A)$} \label{figDefx}
\end{figure}
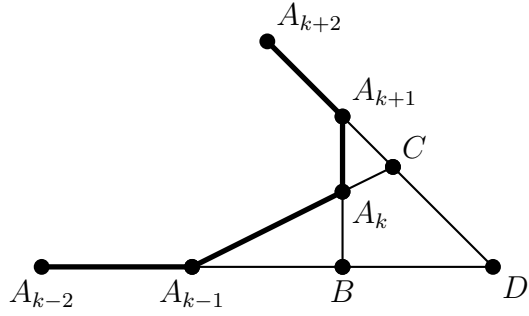

In \cite{G}, we work with related quantities called the $y$-parameters and denoted $y_j$ for $j \in \mathbb{Z}$.  These parameters are defined on $\mathcal{P}_n$ by 
\begin{displaymath}
y_j = \begin{cases}
         -(x_jx_{j+1})^{-1}, & j \textrm{ even} \\
         -x_jx_{j+1},        & j \textrm{ odd} \\
      \end{cases}
\end{displaymath}
and on $\mathcal{P}_n^*$ by 
\begin{displaymath}
y_j = \begin{cases}
         -x_jx_{j+1},        & j \textrm{ even} \\
         -(x_jx_{j+1})^{-1}, & j \textrm{ odd} \\
      \end{cases}
\end{displaymath}
We have that $y_{j+2n}=y_j$ for all $j$ due to the analogous property of the $x_j$.  However, $y_1,\ldots,y_{2n}$ do not give a set of coordinates as they satisfy the single relation $y_1y_2\cdots y_{2n} = 1$.

The $y$-parameters can equivalently be defined using cross ratios (see Figure \ref{figDefy}).  For each index $k$ of $A$
\begin{align} 
y_{2k}(A) &= -\left(\chi (\join{A_{k}}{A_{k-2}}, \join{A_{k}}{A_{k-1}}, \join{A_{k}}{A_{k+1}}, \join{A_k}{A_{k+2}})\right)^{-1} \label{yvertex} \\
y_{2k+1}(A) &= -\chi (B, A_k, A_{k+1}, E) \label{yedge} 
\end{align}
where $B=\meet{\join{A_{k-2}}{A_{k-1}}}{\join{A_k}{A_{k+1}}}$ and $E=\meet{\join{A_k}{A_{k+1}}}{\join{A_{k+2}}{A_{k+3}}}$.

\begin{figure}
\psset{unit=.8cm}
\begin{pspicture}(12,6)
\rput(0,.5){
\psline[showpoints=true,linewidth=2pt](1,1)(3,1)(5,2)(5,3)(4,4)
\uput[d](1,1){$A_{k-2}$}
\uput[d](3,1){$A_{k-1}$}
\uput[dr](5,2){$A_k$}
\uput[r](5,3){$A_{k+1}$}
\uput[ur](4,4){$A_{k+2}$}
\psline(1,1)(5,2)(4,4)
}
\rput(7,0){
\psline[showpoints=true,linewidth=2pt](1,1)(3,1)(5,2)(5,3)(4,4)(2,5)
\uput[d](1,1){$A_{k-2}$}
\uput[d](3,1){$A_{k-1}$}
\uput[r](5,2){$A_k$}
\uput[r](5,3){$A_{k+1}$}
\uput[ur](4,4){$A_{k+2}$}
\uput[ur](2,5){$A_{k+3}$}
\psline[showpoints=true,dotsize=2pt 5](3,1)(5,1)(5,3.5)(4,4)
\uput[dr](5,1){$B$}
\uput[r](5,3.5){$E$}
}
\end{pspicture}
\psset{unit=1cm}
\caption{The lines and points used in the definitions of $y_{2k}(A)$ (left) and $y_{2k+1}(A)$ (right)} \label{figDefy}
\end{figure}
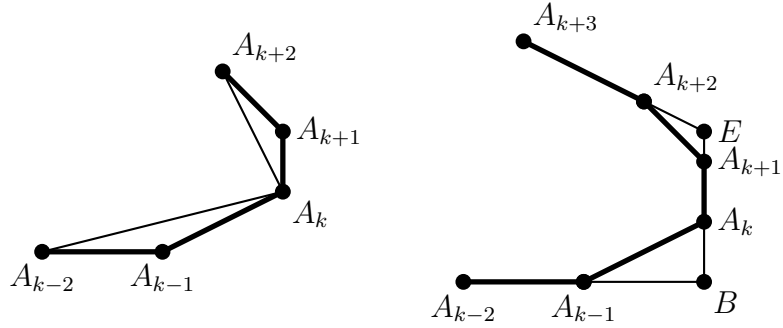

The $y$-parameters transform under the pentagram map according to the $Y$-pattern dynamics of a certain cluster algebra.  We used results of Fomin and Zelevinsky \cite{FZ} to give formulas for the iterates of the pentagram map in terms of the $F$-polynomials $F_{j,k}$ of this cluster algebra.  These can be defined recursively by $F_{j,-1}=F_{j,0} = 1$ and
\begin{equation}
\label{Fjkrec}
F_{j,k+1} = \frac{F_{j-3,k}F_{j+3,k} + M_{j,k}F_{j-1,k}F_{j+1,k}}{F_{j,k-1}}
\end{equation}
for $k \geq 0$ where 
\begin{displaymath}
M_{j,k} = \prod_{i=-k}^k y_{3i+j}
\end{displaymath}

\begin{thm}[{\cite[Theorem 4.2 and Theorem 1.2]{G}}]
Let $A \in \mathcal{P}_n$, $x_j = x_j(A)$, and $y_j = y_j(A)$.  Then 
\begin{align} 
\label{Tkx}
x_j(T^k(A)) &= \begin{cases}
             x_{j-3k}\left(\displaystyle\prod_{i=-k}^{k-1} y_{j+1+3i}\right)\dfrac{F_{j+2,k-1}F_{j-3,k}}{F_{j-2,k-1}F_{j+1,k}}, & j+k \textrm{ even} \\
             x_{j+3k}\left(\displaystyle\prod_{i=-k}^{k-1} y_{j+1+3i}\right)\dfrac{F_{j-3,k-1}F_{j+2,k}}{F_{j+1,k-1}F_{j-2,k}}, & j+k \textrm{ odd} \\
          \end{cases} \\
\label{Tky}
y_j(T^k(A)) &= \begin{cases}
                 M_{j,k}\dfrac{F_{j-1,k}F_{j+1,k}}{F_{j-3,k}F_{j+3,k}}, & j+k \textrm{ even} \\
                 (M_{j,k-1})^{-1} \dfrac{F_{j-3,k-1}F_{j+3,k-1}}{F_{j-1,k-1}F_{j+1,k-1}}, & j+k \textrm{ odd} \\
              \end{cases}
\end{align} 
\end{thm}

The $F$-polynomials are polynomials in the $y_j$ (hence Laurent polynomials in the $x_j$) with positive coefficients.  They have a simple combinatorial description as generating functions of order ideals of certain posets $P_k$, which were first studied by Elkies, Kuperberg, Larsen, and Propp \cite{EKLP}.  Specifically, let $Q_k$ to be the set of triples $(r,s,t)\in \mathbb{Z}^3$ such that $|r| + |s| \leq k-2$, $r+s \equiv k \pmod{2}$, and
\begin{displaymath}
t \in [2|s|-k+2,k-2-2|r|]_4
\end{displaymath}
Note that $Q_k$ and $Q_{k+1}$ are disjoint.  Let $P_k = Q_{k+1} \cup Q_k$.  Define a partial order on $P_k$ by saying that $(r',s',t')$ covers $(r,s,t)$  if and only if $t' = t + 1$ and $|r'-r| + |s'-s| = 1$.  The partial order on $P_k$ restricts to a partial order on $Q_k$.  Let $J(P_k)$ denote the set of order ideals of $P_k$.  The Hasse diagrams for $P_2$ and $P_3$ are given in Figure \ref{figPk}.

\begin{figure}
\begin{pspicture}(11,6)
\rput(1,4){\rnode{v1}{(0,-1,1)}}
\rput(3,4){\rnode{v2}{(0,1,1)}}
\rput(2,3){\rnode{v3}{(0,0,0)}}
\rput(1,2){\rnode{v4}{(-1,0,-1)}}
\rput(3,2){\rnode{v5}{(1,0,-1)}}
\ncline{v4}{v3}
\ncline{v5}{v3}
\ncline{v3}{v1}
\ncline{v3}{v2}

\rput(4,0){
\rput(2,5){\rnode{w1}{(0,-2,2)}}
\rput(4,5){\rnode{w2}{(0,0,2)}}
\rput(6,5){\rnode{w3}{(0,2,2)}}
\rput(3,4){\rnode{w4}{(0,-1,1)}}
\rput(5,4){\rnode{w5}{(0,1,1)}}
\rput(1,3){\rnode{w6}{(-1,-1,0)}}
\rput(3,3){\rnode{w7}{(-1,1,0)}}
\rput(5,3){\rnode{w8}{(1,-1,0)}}
\rput(7,3){\rnode{w9}{(1,1,0)}}
\rput(3,2){\rnode{w10}{(-1,0,-1)}}
\rput(5,2){\rnode{w11}{(1,0,-1)}}
\rput(2,1){\rnode{w12}{(-2,0,-2)}}
\rput(4,1){\rnode{w13}{(0,0,-2)}}
\rput(6,1){\rnode{w14}{(2,0,-2)}}

\ncline{w12}{w10}
\ncline{w13}{w10}
\ncline{w13}{w11}
\ncline{w14}{w11}
\ncline{w10}{w6}
\ncline{w10}{w7}
\ncline{w11}{w8}
\ncline{w11}{w9}
\ncline{w6}{w4}
\ncline{w7}{w5}
\ncline{w8}{w4}
\ncline{w9}{w5}
\ncline{w4}{w1}
\ncline{w4}{w2}
\ncline{w5}{w2}
\ncline{w5}{w3}
}
\end{pspicture}
\caption{The poset $P_k$ for $k=2$ (left) and $k=3$ (right)}
\label{figPk}
\end{figure}
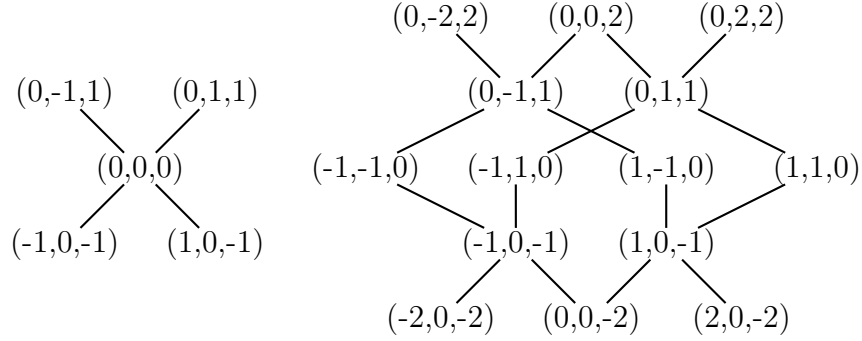

\begin{thm}[{\cite[Theorem 6.6]{G}}]
\begin{equation}
\label{Fjk}
F_{j,k} = \sum_{I \in J(P_k)} \prod_{(r,s,t)\in I} y_{3r+s+j}
\end{equation}
\end{thm}

\section{The $F$-polynomials}
According to \eqref{alpha1x} and \eqref{alpha2x}, the pentagram map has singularities for polygons with $x_jx_{j+1} = 1$, i.e., $y_j = -1$, for some $j$.  According to \eqref{Tkx}, the iterate $T^k$ has a singularity whenever $F_{j,k-1}=0$ or $F_{j,k}=0$ for some $j$.  In this section we examine under which circumstances having $y_j=-1$ for certain $j$ forces an $F$-polynomial to vanish.  Results along these lines will indicate how many steps a given singularity persists.

For the purpose of this section, relax the assumptions $y_{i+2n} = y_i$ for all $i$ and $y_1y_2\cdots y_{2n}=1$.  Instead consider the $F_{j,k}$ as polynomials in the countable collection of variables $\{y_i: i \in \mathbb{Z}\}$.  By way of notation, if $S \subseteq \mathbb{Z}$ let $F_{j,k}|_S$ be the polynomial in $\{y_i: i \in \mathbb{Z} \setminus S\}$ obtained by substituting $y_i=-1$ for all $i \in S$ into $F_{j,k}$.

An \emph{alternating sign matrix} is a square matrix of 1's, 0's, and -1's such that
\begin{itemize}
\item the non-zero entries of each row and column alternate in sign and
\item the sum of the entries of each row and column is 1.
\end{itemize}
Let $ASM(k)$ denote the set of $k$ by $k$ alternating sign matrices.

Elkies, Kuperberg, Larsen, and Propp \cite{EKLP} establish connections between alternating sign matrices and the posets $P_k$ and $Q_k$.  Specifically, they give a bijection from $ASM(k)$ to $J(Q_k)$.  Alternating sign matrices $A \in ASM(k)$ and $B \in ASM(k+1)$ are called compatible if the corresponding order ideals $I \in J(Q_k)$, $J \in J(Q_{k+1})$ have the property that $I \cup J$ is an order ideal of $P_k = Q_k \cup Q_{k+1}$.  Compatible pairs of alternating sign matrices are in bijection with $J(P_k)$.

We will not explicitly state the bijection between $ASM(k)$ and $J(Q_k)$.  Instead, we will list the needed properties of the bijection in the following two lemmas.  Each of these statements can easily be deduced from results of \cite{EKLP}.

\begin{lem}
\label{lemBij1}
Let $A \in ASM(k)$.  Suppose that $A$ has $m$ entries equal to 1, namely entries $(i_1,j_1),(i_2,j_2),\ldots,(i_m,j_m)$.  Then there are $2^m$ alternating sign matrices $B \in ASM(k+1)$ compatible with $A$.  Moreover, there exist an order ideal $J_0 \in J(Q_{k+1})$ and elements $(r_1,s_1,t_1),\ldots,(r_m,s_m,t_m) \in Q_{k+1} \setminus J_0$ such that
\begin{itemize}
\item The map $(i,j) \mapsto (i+j-k-1,-i+j)$ sends $(i_1,j_1),\ldots, (i_m,j_m)$ to $(r_1,s_1),\ldots, (r_m,s_m)$.
\item An order ideal $J \in J(Q_{k+1})$ corresponds to a matrix $B$ compatible with $A$ if and only if
\begin{displaymath}
J_0 \subseteq J \subseteq J_0 \cup \{(r_1,s_1,t_1),\ldots,(r_m,s_m,t_m)\}
\end{displaymath}
\end{itemize}
\end{lem}
By way of notation, let $B_0=B_0(A) \in ASM(k+1)$ be the matrix corresponding to the order ideal $J_0$ from the lemma.

\begin{lem}
\label{lemBij2}
Let $A \in ASM(k)$ and suppose that $a_{11} = a_{kk} = 1$.  This implies $a_{1k} = a_{k1} = 0$.  Let $A' \in ASM(k)$ be identical to $A$ except on the corners where $a'_{11} = a'_{kk} = 0$ and $a'_{1k} = a'_{k1} = 1$.  Let $B = B_0(A)$ and $B' = B_0(A')$.  
\begin{itemize}
\item Let $I, I' \in J(Q_k)$ be the order ideals corresponding to $A$ and $A'$ respectively.  Then $I \subseteq I'$, and $I' \setminus I$ contains exactly one element $(r,s,t)$ for each $r,s$ with $|r|+|s| \leq k-2$ and $r+s \equiv k \pmod{2}$.
\item Let $J, J' \in J(Q_{k+1})$ be the order ideals corresponding to $B$ and $B'$ respectively.  Then $J \subseteq J'$, and $J' \setminus J$ contains exactly one element $(r,s,t)$ for each $r,s$ with $|r|+|s| \leq k-1$, $r+s \equiv k+1 \pmod{2}$, and $s \neq \pm (k-1)$.
\end{itemize}
\end{lem}

Define the weight of an alternating sign matrix $A$ with respect to an integer $j$ to be
\begin{displaymath}
\wt(A,j) = \prod_{(r,s,t)\in I} y_{3r+s+j}
\end{displaymath}
where $I \in J(Q_k)$ is the order ideal corresponding to $A$.  By \eqref{Fjk} we have
\begin{displaymath}
F_{j,k} = \sum_{A,B} \wt(A,j)\wt(B,j)
\end{displaymath}
where the sum is over all compatible pairs $A \in ASM(k)$, $B \in ASM(k+1)$.  Recall that $r+s \equiv k \pmod{2}$ for $(r,s,t) \in Q_k$.  Therefore $\wt(A,j)$ is a monomial in $\{y_i : i \equiv j+k \pmod{2}\}$ while $\wt(B,j)$ is a monomial in $\{y_i : i \equiv j+k+1 \pmod{2}\}$.

\begin{prop}
\begin{equation}
\label{FASM}
F_{j,k} = \sum_{A \in ASM(k)} \wt(A,j)\wt(B_0(A),j)\prod_{a_{il}=1}(1+y_{j+2i+4l-3k-3})
\end{equation}
\end{prop}
\begin{proof}
We need to show for each fixed $A \in ASM(k)$ that
\begin{displaymath}
\sum_B wt(B,j) = \wt(B_0(A),j)\prod_{a_{il}=1}(1+y_{j+2i+4l-3k-3})
\end{displaymath}
where the sum is over $B$ compatible with $A$.  By Lemma \ref{lemBij1}, both sides have $2^m$ terms where $m$ is the number of 1's in $A$.  Moreover, the lowest degree term of both sides is $\wt(B_0(A),j)$.  Let $J_0$ be the order ideal corresponding to $B_0(A)$.  If $a_{il}=1$ then Lemma \ref{lemBij1} says that it is possible to add some $(r,s,t)$ with $r=i+l-k-1$ and $s=l-i$ to $J_0$ to get a new order ideal $J$ corresponding to a matrix $B$ compatible with $A$.  Computing:
\begin{align*}
\wt(B,j) &= \wt(B_0(A),j)y_{3r+s+j} \\
&= \wt(B_0(A),j)y_{j+2i+4l-3k-3}
\end{align*}
By Lemma \ref{lemBij1}, this can be done for arbitrary subsets of the set of 1's of $A$, so the proposition follows.
\end{proof}

Let $S$ be a set of integers.  Say that a matrix $A \in ASM(k)$ \emph{avoids} $(S,j)$ if $j+2i+4l-3k-3 \notin S$ for all $(i,l)$ with $a_{il}=1$.  As $i$ and $l$ range from $1$ to $k$, the index $j+2i+4l-3k-3$ takes on values in the $k$ by $k$ array 
\begin{equation}
\label{iArray}
\left(
\begin{array}{cccc}
j-3k+3 & j-3k+7 & \cdots & j+k-1 \\
j-3k+5 & j-3k+9 & \cdots & j+k+1 \\
\vdots & \vdots & \ddots & \vdots \\
j-k+1 & j-k+5 & \cdots & j+3k-3
\end{array}
\right)
\end{equation}
Hence, $A$ avoids $(S,j)$ if and only if the entries of $A$ equal to 1 avoid the entries of this array contained in $S$.

Let $S$ be a set of integers.  Then we can use \eqref{FASM} to compute $F_{j,k}|_S$ by substituting $y_i=-1$ for all $i \in S$.  If some $A \in ASM(k)$ does not avoid $(S,j)$ then the corresponding term of the sum will have a factor $1+y_i$ for some $i \in S$.  Hence this whole term will vanish in $F_{j,k}|_S$.  

\begin{cor}
\label{corFeq0}
Fix $j$ and $k$, and let $S \subseteq \{i: i \equiv j+k+1 \pmod{2}\}$.  Then
\begin{equation}
\label{FASMRed}
F_{j,k}|_S = \sum_{A} \wt(A,j)(\wt(B_0(A),j))|_S\prod_{a_{il}=1}(1+y_{j+2i+4l-3k-3})
\end{equation}
 where the sum is over those $A \in ASM(k)$ which avoid $(S,j)$.  In particular, if no such $A$ exists then $F_{j,k}|_S \equiv 0$.
\end{cor}

Unfortunately, if there do exist matrices $A \in ASM(k)$ avoiding $(S,j)$, then it is not safe to conclude that $F_{j,k}|_S \not\equiv 0$.  Indeed, recall that $\wt(B_0(A),j)$ is a monomial in $\{y_i : i \equiv j+k+1 \pmod{2}\}$.  When we substitute $y_i=-1$ for $i \in S$, a sign is introduced.  It is therefore possible that different terms of \eqref{FASMRed} cancel with each other.  We can at least conclude that $F_{j,k}|_S \not\equiv 0$ if there exists a unique $A$ avoiding $(S,j)$.

\begin{prop}
\label{Feq0}
Fix $l \in [-(k-1),k-1]_2$.  Let $S \subseteq \mathbb{Z}$ be either 
\begin{enumerate}
\item $[j+l-2(k-1),j+l+2(k-1)]_4$ or 
\item $[j+2l-(k-1),j+2l+(k-1)]_2$.
\end{enumerate}
Then $F_{j,k}|_S \equiv 0$.
\end{prop}
\begin{proof}
Suppose $S = [j+l-2(k-1),j+l+2(k-1)]_4$.  Then $S$ consists precisely of the elements of some row of \eqref{iArray}.  Every $A \in ASM(k)$ must have at least one entry equal to 1 in this row by the definition of alternating sign matrices.  So by Corollary~ \ref{corFeq0}, $F_{j,k}|_S \equiv 0$.  Similarly $S = [j+2l-(k-1),j+2l+(k-1)]_2$ corresponds to a column of \eqref{iArray}, so the same result holds.
\end{proof}

\begin{lem}
Fix $l \in [-(k+1),k+1]_2$ and let
\begin{displaymath}
\sigma = \begin{cases}
           0,  & k \textrm{ odd} \\
           -2, & k \textrm{ even and } l+k \equiv 1 \pmod{4} \\
           2,  & k \textrm{ even and } l+k \equiv 3 \pmod{4} \\
          \end{cases}
\end{displaymath}
Then the set of entries of \eqref{iArray} contained in $\{j+l-2k,j+l+\sigma,j+l+2k\}$ corresponds to the 1's of a permutation matrix.
\label{lemKnight}
\end{lem}
\begin{proof}
Consider a ``knight's path'' in a $k$ by $k$ matrix which starts at some entry in the first column and moves one column right and two rows up each step.  Suppose the rows are ordered cyclically so whenever the knight passes the top of the matrix it wraps around to the bottom.  Continue until the knight reaches the last column placing 1's everywhere it visits.  One possible resulting matrix for $k=7$ is:
\begin{displaymath}
\left(
\begin{array}{ccccccc}
0 & 0 & 0 & 0 & 0 & 1 & 0 \\
0 & 1 & 0 & 0 & 0 & 0 & 0 \\
0 & 0 & 0 & 0 & 1 & 0 & 0 \\
1 & 0 & 0 & 0 & 0 & 0 & 0 \\
0 & 0 & 0 & 1 & 0 & 0 & 0 \\
0 & 0 & 0 & 0 & 0 & 0 & 1 \\
0 & 0 & 1 & 0 & 0 & 0 & 0 \\
\end{array}
\right)
\end{displaymath}

The wrap arounds divide the path into three segments (one of which might have length zero).  The entries of \eqref{iArray} are constant along each segment, and differ by $2k$ between consecutive segments.  If $k$ is odd then the result is a permutation matrix, and the 1's correspond to entries of the array equal to $j+l-2k$, $j+l$, or $j+l+2k$ for some $l$.  One can check that the set of values of $l$ that arise in this manner is precisely $[-(k+1),k+1]_2$.  

If $k$ is even then the resulting matrix will not be a permutation matrix, but one can be obtained by shifting the 1's of the middle segment either up or down by one row.  For instance, in the following example with $k=6$ the middle segment is shifted down:
\begin{displaymath}
\left(
\begin{array}{ccccccc}
1 & 0 & 0 & 1 & 0 & 0 \\
0 & 0 & 0 & 0 & 0 & 0 \\
0 & 0 & 1 & 0 & 0 & 1 \\
0 & 0 & 0 & 0 & 0 & 0 \\
0 & 1 & 0 & 0 & 1 & 0 \\
0 & 0 & 0 & 0 & 0 & 0 \\
\end{array}
\right)
\to
\left(
\begin{array}{ccccccc}
1 & 0 & 0 & 0 & 0 & 0 \\
0 & 0 & 0 & 1 & 0 & 0 \\
0 & 0 & 0 & 0 & 0 & 1 \\
0 & 0 & 1 & 0 & 0 & 0 \\
0 & 0 & 0 & 0 & 1 & 0 \\
0 & 1 & 0 & 0 & 0 & 0 \\
\end{array}
\right)
\end{displaymath}
The effect on the corresponding entry of the array is accounted for by adding $\sigma$.
\end{proof}

\begin{prop}
\label{propFneq0}
Fix $l \in [-(k+1),k+1]_2$ and let
\begin{displaymath}
S = [j-3k+3,j+3k-3]_2 \setminus \{j+l-2k,j+l+\sigma,j+l+2k\} 
\end{displaymath}
for $\sigma$ as in Lemma \ref{lemKnight}.  Let $A$ be the permutation matrix from the conclusion of Lemma \ref{lemKnight}.  Then
\begin{displaymath}
F_{j,k}|_S = \wt(A,j)(\wt(B_0(A),j))|_S(1+y_{j+l-2k})^a(1+y_{j+l+\sigma})^b(1+y_{j+l+2k})^c
\end{displaymath}
for some nonnegative integers $a,b,c$ with $a+b+c = k$.  In particular, $F_{j,k}|_S \not\equiv 0$.
\end{prop}
\begin{proof}
By Lemma \ref{lemKnight}, all entries of \eqref{iArray} are contained in $S$ except those corresponding to the 1's of $A$.  On the one hand, this means that $A$ avoids $(S,j)$.  On the other hand, let $A' \in ASM(k)$ be any matrix avoiding $(S,j)$.  Each column of $A'$ must have at least one entry equal to 1, and 1's can only occur away from elements of $S$.  So each column of $A'$ must have exactly one 1 and it must be in the same place as a 1 of $A$.  As each column of $A'$ has only a single 1, there cannot be any $-1$'s, so in fact $A' = A$.  Hence $A$ is the only element of $ASM(k)$ avoiding $(S,j)$.  The proposition follows from Corollary \ref{corFeq0}, where $a$, $b$, and $c$ are the sizes of the three segments of the knight's path as defined in the proof of Lemma \ref{lemKnight}.
\end{proof}

\begin{cor}
\label{corFneq0Full}
Fix $l \in [-(k+1),k+1]_2$ and let $S$ be any of
\begin{enumerate}
\item $S = [j+l-2k+4,j+l+2k-4]_4$,
\item $S = [j+l-2k+2,j+l+\sigma-2]_2$, or
\item $S = [j+l+\sigma+2,j+l+2k-2]_2$.
\end{enumerate}
for $\sigma$ as above.
Then $F_{j,k}|_S \not\equiv 0$.
\end{cor}
\begin{proof}
Each $S$ is contained in the corresponding one from Proposition \ref{propFneq0}.  As fewer substitutions are made, $F_{j,k}|_S$ remains nonzero.
\end{proof}

\begin{cor}
\label{corFneq0}
Let $S$ be a finite arithmetic sequence such that
\begin{itemize}
\item $|S| < k$, 
\item consecutive terms of $S$ differ by 4 or 2, and
\item the elements of $S$ have the same parity as $j+k+1$.
\end{itemize}
Then $F_{j,k}|_S \not\equiv 0$.
\end{cor}
\begin{proof}
We may assume without loss of generality that $S \subseteq [j-3(k-1),j+3(k-1)]_1$, as those are the indices of the only $y$-variables that $F_{j,k}$ depends on.  In this case, $S$ is contained in some $S$ from Corollary \ref{corFneq0Full} so we still have $F_{j,k}|_S \not\equiv 0$.
\end{proof}

Proposition \ref{Feq0} and Corollary \ref{corFneq0} give a complete picture as to when $F_{j,k}|_S \equiv 0$ for $S$ of the form $[a,b]_4$ or $[a,b]_2$.  We will use these results to prove confinement for certain singularity types in the next section.  There, we will assume that $n$ is large relative to $|S|$ so that the relations among the $y$-variables do not enter into play.  In contrast, the following proposition pertains to a more severe singularity type, so we will reintroduce those relations at this point.

\begin{prop}
\label{propFneq0Big}
Suppose that $n$ is odd and $S = [2,2n-2]_2$.  Assume that $y_{i+2n} = y_i$ for all $i \in \mathbb{Z}$, $y_i = -1$ for all $i \in S$, and $y_1\cdots y_{2n} = 1$.  Let $j,k \in \mathbb{Z}$ with $j+k$ odd and $k \in \{n,n+1\}$.  Then evaluated at this input, $F_{j,k} \neq 0$ provided $y_0 \neq -1$ and $y_i \neq 0$ for all $i$.
\end{prop}
\begin{proof}
First suppose that $k = n$, which implies that $k$ is odd and $j$ is even.  Then in Proposition \ref{propFneq0}, $\sigma = 0$.  Moreover, we can choose $l \in [-(k+1),(k+1)]$ such that $l \equiv -j \pmod{2n}$.  So 
\begin{displaymath}
j+l-2k \equiv j+l+\sigma \equiv j+l+2k \equiv 0 \pmod{2n}
\end{displaymath}
By Proposition \ref{propFneq0} we have
\begin{displaymath}
F_{j,k} = M(1+y_0)^n
\end{displaymath}
for some monomial $M$ in $y_1,\ldots,y_{2n}$.  Since $y_0 \neq -1$ and no $y_i$ is zero, we have that $F_{j,k} \neq 0$.

Now let $k=n+1$, in which case $k$ is even and $j$ is odd.  Since $y_{i+2n} = y_i$ for all $i$, we can consider the entries of \eqref{iArray} modulo $2n$.  For example, if $n=7$, $k=8$, and $j=3$ the result is
\begin{displaymath}
\left(
\begin{array}{cccccccc}
10 & 0 & 4 & 8 & 12 & 2 & 6 & 10\\
12 & 2 & 6 & 10 & 0 & 4 & 8 & 12\\
0 & 4 & 8 & 12 & 2 & 6 & 10 & 0\\
2 & 6 & 10 & 0 & 4 & 8 & 12 & 2\\
4 & 8 & 12 & 2 & 6 & 10 & 0 & 4\\
6 & 10 & 0 & 4 & 8 & 12 & 2 & 6\\
8 & 12 & 2 & 6 & 10 & 0 & 4 & 8\\
10 & 0 & 4 & 8 & 12 & 2 & 6 & 10\\
\end{array}
\right)
\end{displaymath}

Since $n$ is odd, each row and each column contains each of $0,2,4,\ldots, 2n-2$ with a single repeat, namely the first and last entry.  We are interested in matrices $A \in ASM(k)$ whose 1's avoid $S$ and hence all correspond to elements of the array equal to 0.  There are two cases.  If $j \neq n$, then the corner entries of the array are nonzero.  This is the situation in the above example.  There will always be a single row and column in the interior of the array that each start and end with 0.  Every other row and column will have a single 0.  There is a unique alternating sign matrix $A$ avoiding $(S,j)$ in this case.  It has 1's everywhere there is a 0 in the array, and a single -1 where needed.  The matrix $A$ corresponding to the above example is
\begin{displaymath}
\left(
\begin{array}{cccccccc}
0 & 1 & 0 & 0 & 0 & 0 & 0 & 0 \\
0 & 0 & 0 & 0 & 1 & 0 & 0 & 0 \\
1 & -1 & 0 & 0 & 0 & 0 & 0 & 1 \\
0 & 0 & 0 & 1 & 0 & 0 & 0 & 0 \\
0 & 0 & 0 & 0 & 0 & 0 & 1 & 0 \\
0 & 0 & 1 & 0 & 0 & 0 & 0 & 0 \\
0 & 0 & 0 & 0 & 0 & 1 & 0 & 0 \\
0 & 1 & 0 & 0 & 0 & 0 & 0 & 0 \\
\end{array}
\right)
\end{displaymath}
This matrix will always have $n+2$ entries equal to 1, so by Corollary \ref{corFeq0} we have
\begin{displaymath}
F_{j,k} = M(1+y_0)^{n+2} \neq 0
\end{displaymath}
where $M$ is a monomial.

Lastly, if $j = n$ (still assuming $k=n+1$) then the corner entries of the array all equal 0.  In the case $n=7$ this looks like
\begin{displaymath}
\left(
\begin{array}{cccccccc}
0 & 4 & 8 & 12 & 2 & 6 & 10 & 0 \\
2 & 6 & 10 & 0 & 4 & 8 & 12 & 2 \\
4 & 8 & 12 & 2 & 6 & 10 & 0 & 4 \\
6 & 10 & 0 & 4 & 8 & 12 & 2 & 6 \\
8 & 12 & 2 & 6 & 10 & 0 & 4 & 8 \\
10 & 0 & 4 & 8 & 12 & 2 & 6 & 10 \\
12 & 2 & 6 & 10 & 0 & 4 & 8 & 12 \\
0 & 4 & 8 & 12 & 2 & 6 & 10 & 0 \\
\end{array}
\right)
\end{displaymath}
In general, there will be two alternating sign matrices avoiding $(S,j)$, both permutation matrices.  For $n=7$ they are
\begin{displaymath}
\left(
\begin{array}{cccccccc}
1 & 0 & 0 & 0 & 0 & 0 & 0 & 0 \\
0 & 0 & 0 & 1 & 0 & 0 & 0 & 0 \\
0 & 0 & 0 & 0 & 0 & 0 & 1 & 0 \\
0 & 0 & 1 & 0 & 0 & 0 & 0 & 0 \\
0 & 0 & 0 & 0 & 0 & 1 & 0 & 0 \\
0 & 1 & 0 & 0 & 0 & 0 & 0 & 0 \\
0 & 0 & 0 & 0 & 1 & 0 & 0 & 0 \\
0 & 0 & 0 & 0 & 0 & 0 & 0 & 1\\
\end{array}
\right)
\textrm{ and }
\left(
\begin{array}{cccccccc}
0 & 0 & 0 & 0 & 0 & 0 & 0 & 1 \\
0 & 0 & 0 & 1 & 0 & 0 & 0 & 0 \\
0 & 0 & 0 & 0 & 0 & 0 & 1 & 0 \\
0 & 0 & 1 & 0 & 0 & 0 & 0 & 0 \\
0 & 0 & 0 & 0 & 0 & 1 & 0 & 0 \\
0 & 1 & 0 & 0 & 0 & 0 & 0 & 0 \\
0 & 0 & 0 & 0 & 1 & 0 & 0 & 0 \\
1 & 0 & 0 & 0 & 0 & 0 & 0 & 0 \\
\end{array}
\right)
\end{displaymath}
Call these $A$ and $A'$ respectively and let $B = B_0(A)$, $B'=B_0(A')$.  By Corollary \ref{corFeq0} we have
\begin{displaymath}
F_{j,k} = (\wt(A,j)\wt(B,j)+\wt(A',j)\wt(B',j))(1+y_0)^{n+1}
\end{displaymath}
Now, Lemma \ref{lemBij2} determines how the order ideals corresponding to $A$, $A'$, $B$, and $B'$ relate to each other.  Using this result, one can check that the weights are related by
\begin{align*}
\wt(A',j) &= (y_1y_3\cdots y_{2n-1})^n\wt(A,j) \\
\wt(B',j) &= y_0^{n+1}(y_2y_4\cdots y_{2n-2})^{n+2}\wt(B,j)
\end{align*}
Consequently,
\begin{displaymath}
F_{j,k} = M(1+(y_0y_1\cdots y_{2n-1})^ny_0(y_2y_4\cdots y_{2n-2})^2)(1+y_0)^{n+1}
\end{displaymath}
where $M = \wt(A,j)\wt(B,j)$.  Now, $y_0y_1y_2\cdots y_{2n-1} = 1$ and $y_2=y_4= \ldots = y_{2n-2} = -1$ so 
\begin{displaymath}
F_{j,k} = M(1+y_0)(1+y_0)^{n+1} = M(1+y_0)^{n+2} \neq 0
\end{displaymath}
\end{proof}

\section{Singularity patterns}
For $i \in \mathbb{Z}$ let
\begin{displaymath}
X_i = \{A \in \mathcal{P}_n : y_{2i}(A) = -1\}
\end{displaymath}
For $j \in (\frac{1}{2} + \mathbb{Z})$ let
\begin{displaymath}
Y_j = \{A \in \mathcal{P}_n : y_{2j}(A) = -1\}
\end{displaymath}

The reason for the different notation is as follows.

\begin{lem}[{\cite[Lemma 7.2]{G}}]
\label{lemColCon}
Let $A \in \mathcal{P}_n$, $i \in \mathbb{Z}$, $j \in (\frac{1}{2} + \mathbb{Z})$.
\begin{enumerate}
\item $A \in X_i$ if and only if $A_{i-2}$, $A_i$, and $A_{i+2}$ are collinear.
\item $A \in Y_j$ if and only if $a_{j-2}$, $a_j$, and $a_{j+2}$ are concurrent.
\end{enumerate}
\end{lem}

Define in the same way subvarieties $X_j \subseteq \mathcal{P}_n^*$ for $j \in (\frac{1}{2} + \mathbb{Z})$ and $Y_i \subseteq \mathcal{P}_n^*$ for $i \in \mathbb{Z}$.

For $S \subseteq \mathbb{Z}$ or $S \subseteq (\frac{1}{2} + \mathbb{Z})$ let
\begin{align*}
X_S &= \bigcap_{i\in S} X_i \\
Y_S &= \bigcap_{i\in S} Y_i
\end{align*}
For instance, $X_{\{3,5\}} = X_3 \cap X_5$ is the set of twisted polygons $A$ for which $A_1, A_3,A_5,A_7$ are all collinear.  On $X_S$, we have that $y_{2i} = -1$ for all $i \in S$.  Therefore we can replace the $F_{j,k}$ in \eqref{Tkx} with $F_{j,k}|_{2S}$ where $2S = \{2i : i \in S\}$.  If all of these restricted polynomials are nonzero, then the corresponding iterate of the pentagram map is defined generically on $X_S$.

\begin{thm}
\label{thmSingConf}
Let $i,m \in \mathbb{Z}$ with $1 \leq m < n/3-1$.  Let
\begin{align*}
S &= [i-(m-1), i+(m-1)]_2 \\
S' &= \left[i-\frac{1}{2}(m-1), i+\frac{1}{2}(m-1)\right]_1
\end{align*}
Then the map $T^k$ is singular on $X_S$ for $1 \leq k \leq m+1$, but $T^{m+2}$ is nonsingular at generic $A \in X_S$.  Moreover, $T^{m+2}(A) \in Y_{S'}$ for such $A$.
\end{thm}

In words, the theorem says that if $A$ is a twisted polygon such that 
\begin{displaymath}
A_{i-(m+1)},A_{i-(m-1)},A_{i-(m-3)},\ldots,A_{i+(m+1)}
\end{displaymath}
are collinear then $A$ is a singular point of the first $m+1$ iterates of the pentagram map.  Moreover, a generic such $A$ is not a singular point of $T^{m+2}$ and the sides 
\begin{displaymath}
b_{i-\frac{m+3}{2}},b_{i-\frac{m+1}{2}},b_{i-\frac{m-1}{2}},\ldots,b_{i+\frac{m+3}{2}}
\end{displaymath}
of $B=T^{m+2}(A)$ pass alternately through two points.

\begin{lem}
\label{lemSing}
$F_{j,m}|_{2S} \equiv 0$ for all $j \in 2S'$.
\end{lem}
\begin{proof}
Suppose $j \in 2S'$.  Then $j = 2i - l$ for some $l \in [-(m-1),m-1]_2$.  Hence $2S = [j+l-2(m-1),j+l+2(m-1)]_4$.  So by Proposition \ref{Feq0} we have that $F_{j,m}|_{2S} \equiv 0$.
\end{proof}

\begin{lem}
\label{lemConf}
$F_{j,k}|_{2S} \not\equiv 0$ for all $j,k \in \mathbb{Z}$ with $k \in \{m+1,m+2\}$ and $j+k$ odd. 
\end{lem}
\begin{proof}
This would seem to follow immediately from Corollary \ref{corFneq0}.  The only difficulty is that the assumption $y_{i+2n} = y_i$ was relaxed in that section.  Recall that $F_{j,k}$ depends only on those $y_i$ for $i \in [j-3(k-1),j+3(k-1)]_1$, a total of $6k-5$ consecutive variables.  We are assuming that $k \leq m+2 < n/3+1$ so $6k-5 \leq 2n$.  It follows that assuming $y_{i+2n}=y_i$ has no effect on the $y$-variables appearing in $F_{j,k}$.
\end{proof}

\begin{proof}[Proof of Theorem \ref{thmSingConf}]
By Lemma \ref{lemSing} and \eqref{Tkx} we have that $T^{m+1}$ and $T^m$ are singular on $X_S$.  This same lemma applied to smaller $m$ implies that all $T^k$ with $k < m$ are also singular on $X_S$.

Now let $k=m+2$.  Then Lemma \ref{lemConf} shows that none of the factors in the expression \eqref{Tkx} for $T^k$ are identically zero.  Hence \eqref{Tkx} defines $T^{m+2}$ generically on $X_S$.  It remains to show that the image is in $Y_{S'}$.  Let $A \in X_S$ be such that $B = T^{m+2}(A)$ is defined and let $y_j = y_j(A)$ for all $j \in \mathbb{Z}$.  Let $j \in S'$ be given.  By Lemma \ref{lemSing}, $F_{2j,m} = 0$.  Therefore by \eqref{Fjkrec}
\begin{displaymath}
0 = F_{2j-3,m+1}F_{2j+3,m+1} + M_{2j,m+1}F_{2j-1,m+1}F_{2j+1,m+1}
\end{displaymath}
Rearranging and using \eqref{Tky}
\begin{align*}
-1 = (M_{2j,m+1})^{-1}\frac{F_{2j-3,m+1}F_{2j+3,m+1}}{F_{2j-1,m+1}F_{2j+1,m+1}} = y_{2j}(B)
\end{align*}
This is justified because $T^{m+2}(A)$ is defined so the factors being divided by are nonzero.  We have $y_{2j}(B) = -1$ for all $j \in S'$ so $B \in Y_{S'}$ as desired.  
\end{proof}

The roles of $S$ and $S'$ can be interchanged in Theorem \ref{thmSingConf}.  This is apparently an instance of projective duality.
\begin{thm}
\label{thmSingConf2}
Let $S$ and $S'$ be as in Theorem \ref{thmSingConf}.  Then the map $T^k$ is singular on $X_{S'}$ for $1 \leq k \leq m+1$, but $T^{m+2}$ is nonsingular at generic $A \in X_{S'}$.  Moreover, $T^{m+2}(A) \in Y_{S}$ for such $A$.
\end{thm}
\begin{proof}
The proof is essentially obtained by switching $S$ and $S'$ throughout in the proofs of Theorem $\ref{thmSingConf}$ and its lemmas.  The only difference is that the proof now utilizes different cases of Proposition \ref{Feq0} and Corollary \ref{corFneq0}, namely the cases involving arithmetic sequences with common difference 2.
\end{proof}

We now have singularity confinement on $X_S$ for $S$ an arithmetic sequence whose terms differ by 1 or 2.  Generally, if $S$ is a disjoint union of such sequences which are far apart from each other, then the corresponding singularities do not affect each other.  Hence singularity confinement holds and the number of steps needed to get past the singularity is dictated by the length of the largest of the disjoint sequences.

Not all singularity types are of this form.  For instance, consider $S=\{3,4,7,8\}$.  By Theorem \ref{thmSingConf2}, $T^4$ is defined generically on both $X_{\{3,4\}}$ and $X_{\{7,8\}}$.  One can check that $T^4$ is singular on $X_S$, although another step does suffice to move past the singularity.  For general types $S$, it is difficult to predict how many steps the corresponding singularities last.  However, it would seem that singularity confinement does hold outside of some exceptional cases.  

If $n$ is odd, the only exceptional type is $S=[1,n]_1$.  Moreover, for any other $S$ and generic $A \in X_S$ the corresponding singularity lasts at most $n$ steps.  We establish this by considering the worst case where $|S| = n-1$.

\begin{prop}
\label{propSingConf}
Suppose $n$ is odd and let $S = [1,n]_1 \setminus \{i\}$ for some $i \in [1,n]_1$.  Then $T^{n+1}$ is nonsingular at generic $A \in X_S$.
\end{prop}
\begin{proof}
Suppose without loss of generality that $i=n$.  Then $y_2=y_4=\ldots=y_{2n-2} = -1$.  As always, we have $y_{j+2n}=y_j$ for all $j$ and $y_1y_2\cdots y_{2n}=1$.  So Proposition \ref{propFneq0Big} implies that, generically, $F_{j,k} \neq 0$ for $j+k$ odd and $k \in \{n,n+1\}$.  Therefore, $T^{n+1}$ is generically defined by \eqref{Tkx}.  
\end{proof}
\begin{rem}
In fact, Proposition \ref{propFneq0Big} says more, namely that the relevant $F$-polynomials never vanish unless $y_0=-1$ or some $y_i=0$.  The assumption that quadruples of consecutive vertices be in general position forces all of the $y_i$ to be nonzero.  As such, we have that the only singularities of $T^{n+1}$ on $X_S$ occur when $y_{2n}=y_0=-1$.  Hence, $T^{n+1}$ restricts to a regular map on $X_S \setminus X_{[1,n]_1}$.
\end{rem}

\begin{cor}
Suppose that $n$ is odd and that $S \subsetneq [1,n]_1$.  Then $T^{n+1}$ is nonsingular for generic $A \in X_S$.
\end{cor}
\begin{proof}
Since $S \neq [1,n]_1$, there exists some $S'$ such that $S \subseteq S' \subsetneq [1,n]_1$ and $|S'|=n-1$.  Since $S \subseteq S'$ we have $X_S \supseteq X_{S'}$.  Now $T^{n+1}$ is nonsingular at generic $A \in X_{S'}$ by Proposition \ref{propSingConf}.  In particular, the map is defined at some such $A$, which is necessarily also in $X_S$.  It follows that $T^{n+1}$ is defined generically on $X_S$.
\end{proof}

Of course for general $S$, it will usually be the case that $T^m$ is defined on $X_S$ for some $m < n+1$.  The corollary only ensures that $n+1$ steps will be sufficient.  This appears to also be true for $n$ even outside of some exceptional cases.  We state this as a conjecture.

\begin{conj} \  
\label{conjSingConf}
Suppose that $n$ is even.
\begin{itemize}
\item Singularity confinement holds generically on $X_S$ unless $[1,n-1]_2 \subseteq S$ or $[2,n]_2 \subseteq S$.
\item Whenever singularity confinement holds for a type, there exists an $m \leq n$ such that generic singularities of that type last $m$ steps (i.e. $T^m$ is singular but $T^{m+1}$ is not).
\end{itemize}
\end{conj}
\begin{rem}
The cases where singularity confinement fails to hold are quite extreme.  If $S=[1,n]_1$ then $A \in X_S$ has all its vertices lying on two lines.  It follows that all the vertices of $T(A)$ are equal.  If $n$ is even and say $S$ contains $[1,n-1]_2$ then half the vertices of $A \in X_S$ are collinear and $T(A)$ will be contained in the common line.  Amazingly, if $n$ is even and $A \in X_S$ for $S=[1,n]_1$ then a finite number of iterations of $T^{-1}$ takes $A$ to $Y_S$ \cite[Theorem 3]{S}, \cite[Theorem 7.9]{G}.  Similar results likely hold for the other exceptional singularity types.
\end{rem}

\section{Straightedge constructions: a first attempt}
Let $A \in \mathcal{P}_n$ be a singular point of $T^k$ for $1 \leq k < m$ but not of $T^m$.  The remainder of this paper focuses on the problem of constructing $B=T^m(A)$.

One possible approach would be to compute the $x$-coordinates of $A$, plug into \eqref{Tkx} and \eqref{Fjk} to find the $x$-coordinates of $B$, and then use these to construct $B$ itself.  This process would be computationally expensive as the number of terms of $F_{j,k}$ grows superexponentially with $k$.  More to the point, this approach has the drawback that it ignores the geometry of the pentagram map.

As an alternative, we could choose a one-parameter family $A(t)$ of twisted polygons varying continuously with $t$ such that 
\begin{enumerate}
\item $A(0)=A$ and 
\item $A(t)$ is a regular point of $T^k$ for all $t\neq 0$ and $k \leq m$.
\end{enumerate}
For small $t \neq 0$, we can obtain $B(t) = T^m(A(t))$ by iterating the geometric construction defining $T$.  By continuity, $B$ is given by $\lim_{t\to 0} B(t)$ which can be found numerically.  This method is perhaps more feasible, but it involves a limiting procedure.  More satisfying would be a finite construction, preferably one which can be carried out with a straightedge alone, as is the case with the pentagram map.

In this section we introduce an iterative approach to finding such a straightedge construction, which works in simple situations.  The idea is to attempt to make sense of the polygon $T^k(A)$ for $k<m$ despite the presence of the singularity.  Let $A(t)$ be as above, and fixing $k<m$, let $C(t) = T^k(A(t))$.  For each appropriate index $i$, let
\begin{displaymath}
C_i = \lim_{t \to 0} (C_i(t))
\end{displaymath}
We say that $C_i$ is well-defined if this limit always exists and is independent of the choice of the curve $A(t)$ through $A$.  We can define sides $c_j$ of $T^k(A)$ in the same way.  In fact it is possible that each of the $C_i$ and $c_j$ are well-defined, despite the singularity.  This would simply indicate that the resulting polygon $C$ fails to satisfy the property that quadruples of consecutive vertices be in general position, which is needed for all the $x$-coordinates to be defined.

As before, suppose $A \in \mathcal{P}_n$ is a singular point of $T^k$ for $1 \leq k < m$ but not of $T^m$.  In addition, assume that all of the vertices and sides of $T^k(A)$ for $1\leq k<m$ are well-defined.  Then it should be possible to construct the components of these intermediate polygons successively.  Ideally, each individual side or vertex can be constructed by a simple procedure depending only on nearby objects.  

The most basic of these local rules is the usual definition of the pentagram map, namely, if $B=T(A)$ then
\begin{displaymath}
b_i = \join{A_{i-1}}{A_{i+1}}
\end{displaymath}
for each index $i$ of $A$ and
\begin{displaymath}
B_j = \meet{b_{j-\frac{1}{2}}}{b_{j+\frac{1}{2}}}
\end{displaymath}
for each index $j$ of $B$.
These rules only work when $A_{i-1} \neq A_{i+1}$ and $b_{j-\frac{1}{2}} \neq b_{j+\frac{1}{2}}$ respectively.  Other rules are needed to handle other cases.  The next simplest rule involves triple ratios which are a six point analogue of cross ratios.

\begin{defin}
Let $A,B,C,D,E,F$ be points in the plane with $A,B,C$ collinear, $C,D,E$ collinear, and $E,F,A$ collinear.  The \emph{triple ratio} of these points is defined to be
\begin{displaymath}
[A,B,C,D,E,F] = \frac{AB}{BC}\frac{CD}{DE}\frac{EF}{FA}
\end{displaymath}
where for instance, $\frac{AB}{BC}$ refers to the ratio of these two lengths, taken to be positive if $B$ lies between $A$ and $C$ and negative otherwise.
\end{defin}

We will need to following properties of triple ratios, which can be found for instance in \cite{R}.

\begin{prop} \ 
\begin{itemize}
\item Triple ratios are invariant under projective transformations.
\item (Ceva's theorem) If the lines $\join{A}{D}$, $\join{C}{F}$, and $\join{E}{B}$ are concurrent then 
\begin{displaymath}
[A,B,C,D,E,F] = 1
\end{displaymath}
\item (Menelaus' theorem) If $B$, $D$, and $F$, are collinear then 
\begin{displaymath}
[A,B,C,D,E,F] = -1
\end{displaymath}
\end{itemize}
\end{prop}

\begin{prop}
Suppose $A \in \mathcal{P}_n$ is a regular point of $T$ and $T^2$, and let $B=T(A)$, $C=T^2(A)$.  Then for all $i \in \mathbb{Z}$
\begin{displaymath}
[B_{i-\frac{3}{2}},B_{i-\frac{1}{2}},A_i,B_{i+\frac{3}{2}},B_{i+\frac{1}{2}},C_i] = -1
\end{displaymath}
\end{prop} 

\begin{proof}
That the triple ratio makes sense, and that it satisfies the condition of Menelaus' theorem, are both clear from Figure \ref{figDecVert}.
\end{proof}

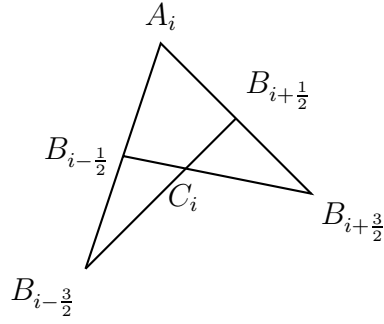
\begin{figure}
\begin{pspicture}(6,5)
\psline(1,1)(2,4)(4,2)(1.5,2.5)
\psline(1,1)(3,3)
\uput[dl](1,1){$B_{i-\frac{3}{2}}$}
\uput[l](1.5,2.5){$B_{i-\frac{1}{2}}$}
\uput[u](2,4){$A_i$}
\uput[dr](4,2){$B_{i+\frac{3}{2}}$}
\uput[ur](3,3){$B_{i+\frac{1}{2}}$}
\uput[d](2.3,2.3){$C_i$}
\end{pspicture}
\caption{By Menelaus' theorem, $[B_{i-\frac{3}{2}},B_{i-\frac{1}{2}},A_i,B_{i+\frac{3}{2}},B_{i+\frac{1}{2}},C_i] = -1$ for any index $i$ of $A$.}
\label{figDecVert}
\end{figure}

Now imagine continuously deforming the polygons until all six of these points are collinear.  This relation holds as the polygons are being deformed, so it continues to hold in the limit.  In particular, if five of these points are well-defined and collinear, then generically the sixth is also well-defined and is the unique point on the line for which the relation holds.  Stated as a rule, if $B=T(A)$, $C=T(B)$, and if $A_i,B_{i-\frac{3}{2}},B_{i-\frac{1}{2}},B_{i+\frac{1}{2}},B_{i+\frac{3}{2}}$ are collinear for some index $i$ of $A$ then
\begin{equation}
\label{ruleTripConj}
C_i = \textrm{\TripleConjugate}(B_{i-\frac{3}{2}},B_{i-\frac{1}{2}},A_i,B_{i+\frac{3}{2}},B_{i+\frac{1}{2}})
\end{equation}
Here, \TripleConjugate\ is a function that inputs five collinear points $P_1,P_2,P_3,P_4,P_5$ and outputs the unique point $P_6$ on the common line such that $[P_1,P_2,P_3,P_4,P_5,P_6]= -1$.

The ordered sextuple of sides $a_j,b_{j-\frac{3}{2}},b_{j+\frac{1}{2}},b_{j-\frac{1}{2}},b_{j+\frac{3}{2}},c_j$ satisfy the same incidences as the vertices $A_i,B_{i-\frac{3}{2}},B_{i-\frac{1}{2}},B_{i+\frac{1}{2}},B_{i+\frac{3}{2}},C_i$ (see Figure \ref{figDecSide}).  So we get the analogous rule, that if for some index $j$ of $B$ the sides $a_j,b_{j-\frac{3}{2}},b_{j+\frac{1}{2}},b_{j-\frac{1}{2}},b_{j+\frac{3}{2}}$ are all concurrent then
\begin{equation}
\label{ruleTripConjDual}
c_j = \textrm{\TripleConjugate}(b_{j-\frac{3}{2}},b_{j+\frac{1}{2}},a_j,b_{j+\frac{3}{2}},b_{j-\frac{1}{2}})
\end{equation}

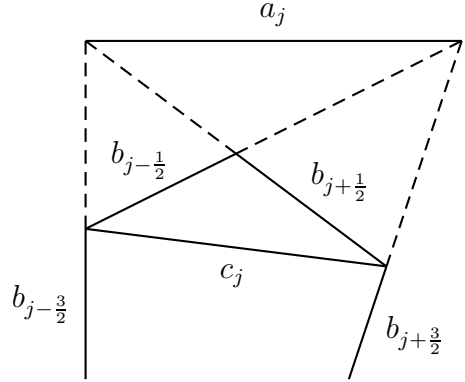
\begin{figure}
\begin{pspicture}(8,6)
\pnode(2,1){A}
\pnode(2,3){B}
\pnode(4,4){C}
\pnode(6,2.5){D}
\pnode(5.5,1){E}
\pnode(2,5.5){F}
\pnode(7,5.5){G}

\ncline{A}{B}
\naput{$b_{j-\frac{3}{2}}$}
\ncline{B}{C}
\naput[labelsep=2pt]{$b_{j-\frac{1}{2}}$}
\ncline{C}{D}
\naput[labelsep=2pt]{$b_{j+\frac{1}{2}}$}
\ncline{D}{E}
\naput{$b_{j+\frac{3}{2}}$}
\ncline{F}{G}
\naput{$a_{j}$}
\ncline{B}{D}
\nbput{$c_{j}$}
\ncline[linestyle=dashed]{F}{B}
\ncline[linestyle=dashed]{F}{C}
\ncline[linestyle=dashed]{G}{C}
\ncline[linestyle=dashed]{G}{D}
\end{pspicture}
\caption{A configuration of lines dual to the configuration of points in Figure \ref{figDecVert}}.
\label{figDecSide}
\end{figure}

These rules are already enough to handle singularities of the simplest type.  The full construction is described in the next subsection.  The following subsection explains the difficulty in handling more complicated singularities.

\subsection{The map $T^3: X_3 \to Y_3$}
The case $m=1$ of Theorem \ref{thmSingConf} says that $T^3$ restricts to a rational map from $X_i$ to $Y_i$.  Assume without loss of generality that $i=3$.  Throughout this subsection, assume $A \in X_3$ and let $B=T(A)$, $C=T^2(A)$, and $D=T^3(A)$.  Since $A \in X_3$ we have that $A_1,A_3,A_5$ are collinear, so let $l$ denote the line containing them.  In the following, $i$ and $j$ will denote elements of $\mathbb{Z}$ and $\frac{1}{2} + \mathbb{Z}$ respectively.

We start by constructing $B$, which can be done via the usual pentagram map.  So construct $b_i = \join{A_{i-1}}{A_{i+1}}$ for all $i$ and $B_j = \meet{b_{j-\frac{1}{2}}}{b_{j+\frac{1}{2}}}$ for all $j$.  Note that $b_2=b_4=l$.  Therefore, $B_{1.5},B_{2.5},B_{3.5},B_{4.5}$ all lie on $l$ and moreover 
\begin{displaymath}
B_{2.5} = B_{3.5} = \meet{l}{b_3}
\end{displaymath}
Let $P$ be this common point.  The construction of $B$ is shown in Figure \ref{figAB1}.

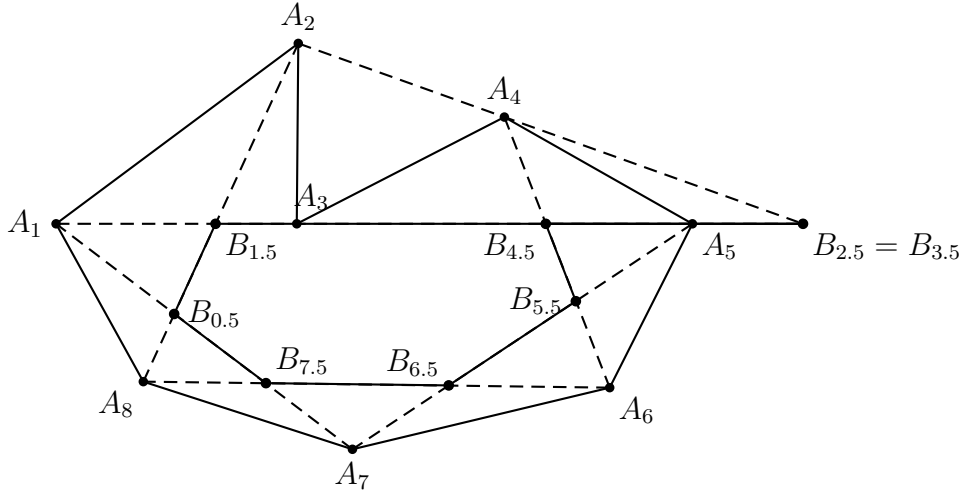
\begin{figure}
\begin{pspicture}(12,7)

\pnode(1,4){A1}
\pnode(4.22,6.4){A2}
\pnode(4.2,4){A3}
\pnode(6.96,5.42){A4}
\pnode(9.46,4){A5}
\pnode(8.36,1.82){A6}
\pnode(4.94,1){A7}
\pnode(2.16,1.9){A8}

\cnode*(2.57,2.8){2pt}{B05}
\cnode*(3.12,4){2pt}{B15}
\cnode*(10.93,4){2pt}{P}
\cnode*(7.51,4){2pt}{B45}
\cnode*(7.91,2.97){2pt}{B55}
\cnode*(6.22,1.85){2pt}{B65}
\cnode*(3.79,1.88){2pt}{B75}

\uput[l](A1){$A_1$}
\uput[u](A2){$A_2$}
\uput[65](A3){$A_3$}
\uput[u](A4){$A_4$}
\uput[dr](A5){$A_5$}
\uput[dr](A6){$A_6$}
\uput[d](A7){$A_7$}
\uput[dl](A8){$A_8$}

\uput[r](B05){$B_{0.5}$}
\uput[dr](B15){$B_{1.5}$}
\uput[dr](P){$B_{2.5}=B_{3.5}$}
\uput[dl](B45){$B_{4.5}$}
\uput[l](B55){$B_{5.5}$}
\uput[ul](B65){$B_{6.5}$}
\uput[ur](B75){$B_{7.5}$}

\pspolygon[showpoints=true](A1)(A2)(A3)(A4)(A5)(A6)(A7)(A8)
\ncline[linestyle=dashed]{A1}{P}
\ncline[linestyle=dashed]{A2}{P}
\ncline[linestyle=dashed]{A4}{A6}
\ncline[linestyle=dashed]{A5}{A7}
\ncline[linestyle=dashed]{A6}{A8}
\ncline[linestyle=dashed]{A7}{A1}
\ncline[linestyle=dashed]{A8}{A2}
\pspolygon[showpoints=true](B05)(B15)(P)(B45)(B55)(B65)(B75)
\end{pspicture}
\caption{The construction of $B=T(A)$ from $A$ for $A \in X_3$}
\label{figAB1}
\end{figure}

Generically, $B_{j-1} \neq B_{j+1}$ for all $j$, so the sides of $C$ can all be constructed as $c_j = \join{B_{j-1}}{B_{j+1}}$.  Note that $c_{2.5} = c_{3.5} = l$.  Hence, we cannot use $\meet{c_{2.5}}{c_{3.5}}$ to construct $C_3$.  However, $A_3, B_{1.5},B_{2.5},B_{3,5},B_{4.5}$ all lie on $l$ so the rule \eqref{ruleTripConj} applies:
\begin{displaymath}
C_3 = \textrm{\TripleConjugate}(B_{1.5},B_{2.5},A_3,B_{4.5},B_{3.5})
\end{displaymath}
As usual, $C_i = \join{c_{i-\frac{1}{2}}}{c_{i+\frac{1}{2}}}$ for all $i \neq 3$.  In particular we have
\begin{align*}
C_2 &= \meet{c_{1.5}}{c_{2.5}} = \meet{\join{B_{0.5}}{B_{2.5}}}{l} = B_{2.5} = P \\
C_4 &= \meet{c_{3.5}}{c_{4.5}} = \meet{l}{\join{B_{3.5}}{B_{5.5}}} = B_{3.5} = P 
\end{align*}
Figure \ref{figBC1} shows the construction of $C$ from $B$.

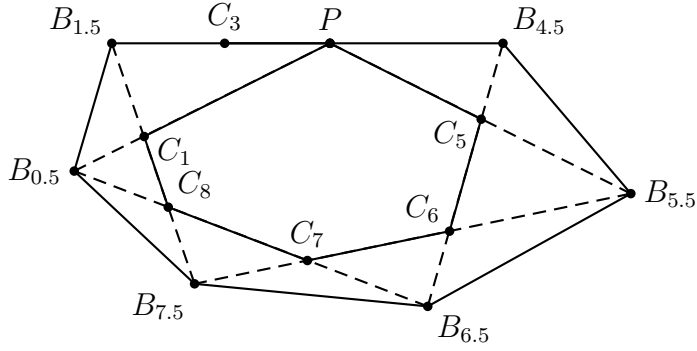
\begin{figure}
\begin{pspicture}(10,5)
\pnode(1,2.3){B05}
\pnode(1.5,4){B15}
\pnode(4.4,4){P}
\pnode(6.7,4){B45}
\pnode(8.4,2){B55}
\pnode(5.7,.5){B65}
\pnode(2.6,.8){B75}

\pnode(1.93,2.76){C1}
\pnode(3,4){C3}
\pnode(6.41,2.99){C5}
\pnode(5.99,1.5){C6}
\pnode(4.1,1.11){C7}
\pnode(2.25,1.82){C8}

\uput[l](B05){$B_{0.5}$}
\uput[ul](B15){$B_{1.5}$}
\uput[u](P){$P$}
\uput[ur](B45){$B_{4.5}$}
\uput[r](B55){$B_{5.5}$}
\uput[dr](B65){$B_{6.5}$}
\uput[dl](B75){$B_{7.5}$}

\uput[330](C1){$C_1$}
\uput[u](C3){$C_3$}
\uput[210](C5){$C_5$}
\uput[ul](C6){$C_6$}
\uput[u](C7){$C_7$}
\uput[ur](C8){$C_8$}

\pspolygon[showpoints=true](B05)(B15)(P)(B45)(B55)(B65)(B75)
\ncline[linestyle=dashed]{B05}{P}
\ncline[linestyle=dashed]{P}{B55}
\ncline[linestyle=dashed]{B45}{B65}
\ncline[linestyle=dashed]{B55}{B75}
\ncline[linestyle=dashed]{B65}{B05}
\ncline[linestyle=dashed]{B75}{B15}
\pspolygon[showpoints=true](C1)(P)(C5)(C6)(C7)(C8)
\psline[showpoints=true](P)(C3)
\end{pspicture}
\caption{The construction of $C=T^2(A)$ from $B=T(A)$ for $A \in X_3$.  Here, $P=B_{2.5}=B_{3.5}=C_2=C_4$.}
\label{figBC1}
\end{figure}

The last difficulty is in constructing the side $d_3$ since $C_2 = C_4$.  However, $b_3,c_{1.5},c_{2.5},c_{3.5},c_{4.5}$ all pass through $P$ so by \eqref{ruleTripConjDual}
\begin{displaymath}
d_3 = \textrm{\TripleConjugate}(c_{1.5},c_{3.5},b_3,c_{4.5},c_{2.5})
\end{displaymath}
In particular, $d_3$ contains $P$.  Letting $d_i = \meet{C_{i-1}}{C_{i+1}}$ for all $i \neq 3$, we have that $d_1$ contains $C_2=P$ and $d_5$ contains $C_4 = P$ as well.  This verifies that $D \in Y_3$.  Finally, the vertices of $D$ are constructed as $D_j = \meet{d_{j-\frac{1}{2}}}{d_{j+\frac{1}{2}}}$ for all $j$.  The construction of $D$ from $C$ is given in Figure \ref{figCD1}.  The full construction of $D$ from $A$ is summarized in Algorithm \ref{algT3}.

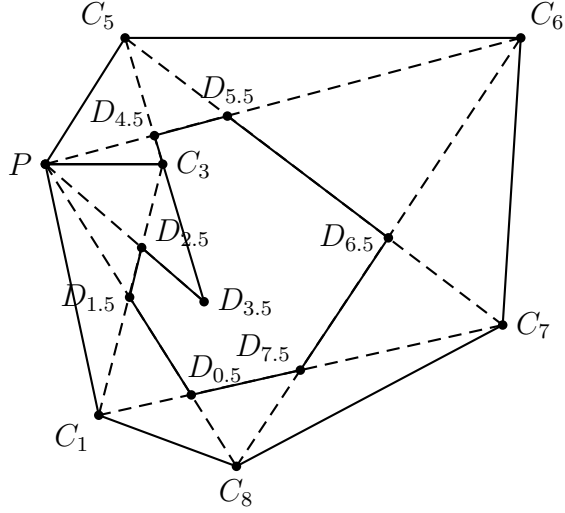
\begin{figure}
\begin{pspicture}(8,7)
\pnode(1.71,1.16){C1}
\pnode(1,4.5){P}
\pnode(2.56,4.5){C3}
\pnode(2.06,6.18){C5}
\pnode(7.32,6.18){C6}
\pnode(7.08,2.36){C7}
\pnode(3.54,.48){C8}

\pnode(2.94,1.43){D05}
\pnode(2.12,2.73){D15}
\pnode(2.28,3.39){D25}
\pnode(3.11,2.67){D35}
\pnode(2.45,4.88){D45}
\pnode(3.42,5.14){D55}
\pnode(5.56,3.52){D65}
\pnode(4.39,1.76){D75}

\uput[dl](C1){$C_1$}
\uput[l](P){$P$}
\uput[r](C3){$C_3$}
\uput[ul](C5){$C_5$}
\uput[ur](C6){$C_6$}
\uput[r](C7){$C_7$}
\uput[d](C8){$C_8$}

\uput[60](D05){$D_{0.5}$}
\uput[l](D15){$D_{1.5}$}
\uput[30](D25){$D_{2.5}$}
\uput[r](D35){$D_{3.5}$}
\uput[ul](D45){$D_{4.5}$}
\uput[u](D55){$D_{5.5}$}
\uput[l](D65){$D_{6.5}$}
\uput[ul](D75){$D_{7.5}$}

\pspolygon[showpoints=true](C1)(P)(C3)(P)(C5)(C6)(C7)(C8)
\ncline[linestyle=dashed]{C1}{C3}
\ncline[linestyle=dashed]{D35}{P}
\ncline[linestyle=dashed]{C3}{C5}
\ncline[linestyle=dashed]{P}{C6}
\ncline[linestyle=dashed]{C5}{C7}
\ncline[linestyle=dashed]{C6}{C8}
\ncline[linestyle=dashed]{C7}{C1}
\ncline[linestyle=dashed]{C8}{P}
\pspolygon[showpoints=true](D05)(D15)(D25)(D35)(D45)(D55)(D65)(D75)
\end{pspicture}
\caption{The construction of $D=T^3(A)$ from $C=T^2(A)$ for $A \in X_3$.  Here, $P=C_2=C_4$.}
\label{figCD1}
\end{figure}

\begin{algorithm}
\caption{$T^3(A)$}
\begin{algorithmic}
\REQUIRE $A_1,A_3,A_5$ collinear
\FORALL{$i$}
\STATE $b_i := \join{A_{i-1}}{A_{i+1}}$  
\ENDFOR
\FORALL{$j$}
\STATE $B_j := \meet{b_{j-\frac{1}{2}}}{b_{j+\frac{1}{2}}}$ 
\ENDFOR
\FORALL{$j$}
\STATE $c_j := \join{B_{j-1}}{B_{j+1}}$ 
\ENDFOR
\FORALL{$i \neq 3$}
\STATE $C_i := \meet{c_{i-\frac{1}{2}}}{c_{i+\frac{1}{2}}}$ 
\ENDFOR
\STATE $C_3 := \textrm{\TripleConjugate}(B_{1.5},B_{2.5},A_3,B_{4.5},B_{3.5})$
\FORALL{$i \neq 3$}
\STATE $d_i := \join{C_{i-1}}{C_{i+1}}$ 
\ENDFOR
\STATE $d_3 := \textrm{\TripleConjugate}(c_{1.5},c_{3.5},b_3,c_{4.5},c_{2.5})$
\FORALL{$j$}
\STATE $D_j := \meet{d_{j-\frac{1}{2}}}{d_{j+\frac{1}{2}}}$ 
\ENDFOR
\RETURN $D$
\end{algorithmic}
\label{algT3}
\end{algorithm}

\begin{rem}
All algorithms in this paper can be carried out as straightedge constructions.  When this is not completely apparent from the algorithm itself, more details are provided in the surrounding text and/or in Appendix A.  For example, Algorithm \ref{algT3} uses the function \TripleConjugate, a construction for which is given in Algorithm \ref{algTripConj} of the Appendix.  In addition to the usual operations of projective geometry (finding a line through two points or the intersection point of two lines) we assume as a primitive a function \RandomPoint() which returns the next in an arbitrarily long sequence $P_1,P_2,\ldots$ of points in the plane.  All algorithms are only claimed to behave correctly for generic choices of these points.  For convenience, define
\begin{itemize}
\item \RandomLine()  :=  $\join{\textrm{ \RandomPoint()}}{\textrm{\RandomPoint() }}$  
\item \RandomPointOn($l$)  :=  $\meet{\textrm{\RandomLine()}}{l}$
\item \RandomLineThrough($P$)  :=  $\join{\textrm{ \RandomPoint()}}{P}$
\end{itemize}
\end{rem}

\subsection{The map $T^4: X_{\{3,5\}} \to Y_{\{3.5,4.5\}}$} \label{secStraightEdge2}
The next simplest case, $m=2$, of Theorem~\ref{thmSingConf} concerns a singularity which disappears after four steps.  Specifically, taking $i=4$ there is a map $T^4: X_{\{3,5\}} \to Y_{\{3.5,4.5\}}$.  Suppose $A \in X_{\{3,5\}}$ which means that $A_1$, $A_3$, $A_5$, and $A_7$ are collinear.  Let $l$ be their common line.  Then $E = T^4(A) \in Y_{\{3.5,4.5\}}$, i.e., $e_{1.5},e_{3.5},e_{5.5}$ are concurrent and $e_{2.5},e_{4.5},e_{6.5}$ are also concurrent.  As before, we will attempt to successively construct the intermediate polygons, namely $B=T(A)$, $C=T^2(A)$, and $D=T^3(A)$.  However, in this case not all of these polygons will be completely well-defined.  Again, let $i$ and $j$ range over $\mathbb{Z}$ and $\frac{1}{2} + \mathbb{Z}$ respectively.  The constructions that follow are illustrated in Figures \ref{figAB2}--\ref{figDE2}. 

\begin{figure}
\begin{pspicture}(9,6)
\pnode(0.88, 3.5){A1}
\pnode(2.3, 5.16){A2}
\pnode(2.88, 3.5){A3}
\pnode(4.38, 2.74){A4}
\pnode(6.74, 3.5){A5}
\pnode(8.46, 5.06){A6}
\pnode(8.7, 3.5){A7}
\pnode(5.74, 0.66){A8}
\pnode(2.34, 1.1){A9}

\pnode(2.32, 2.66){B05}
\pnode(2.32, 3.5){B15}
\pnode(3.73, 3.5){P}
\pnode(5.72, 3.5){Q}
\pnode(7.5, 3.5){B65}
\pnode(7.13, 2.91){B75}
\pnode(3.95, 1.71){B85}

\uput[l](A1){$A_1$}
\uput[u](A2){$A_2$}
\uput[d](A3){$A_3$}
\uput[d](A4){$A_4$}
\uput{2pt}[120](A5){$A_5$}
\uput[u](A6){$A_6$}
\uput[r](A7){$A_7$}
\uput[dr](A8){$A_8$}
\uput[dl](A9){$A_9$}

\uput{2pt}[dl](B05){$B_{0.5}$}
\uput[ul](B15){$B_{1.5}$}
\uput[ur](P){$P$}
\uput[ul](Q){$Q$}
\uput[30](B65){$B_{6.5}$}
\uput{2pt}[ul](B75){$B_{7.5}$}
\uput[d](B85){$B_{8.5}$}

\pspolygon[showpoints=true](A1)(A2)(A3)(A4)(A5)(A6)(A7)(A8)(A9)
\ncline[linestyle=dashed]{A1}{A7}
\ncline[linestyle=dashed]{A2}{A4}
\ncline[linestyle=dashed]{A4}{A6}
\ncline[linestyle=dashed]{A6}{A8}
\ncline[linestyle=dashed]{A7}{A9}
\ncline[linestyle=dashed]{A8}{A1}
\ncline[linestyle=dashed]{A9}{A2}
\pspolygon[showpoints=true](B05)(B15)(P)(Q)(B65)(B75)(B85)

\end{pspicture}
\caption{The construction of $B=T(A)$ from $A$ for $A \in X_{\{3,5\}}$.  Here, $P=B_{2.5}=B_{3.5}$ and $Q=B_{4.5}=B_{5.5}$.}
\label{figAB2}
\end{figure}

\begin{figure}
\begin{pspicture}(8,5)
\pnode(0.54,2.22){B05}
\pnode(1.2, 4.5){B15}
\pnode(2.68,4.5){P}
\pnode(6.96,4.5){Q}
\pnode(8.56, 4.5){B65}
\pnode(8.32, 1.58){B75}
\pnode(3.52, 0.44){B85}

\pnode(1.76, 3.52){C1}
\pnode(3.7, 4.5){C3}
\pnode(4.68, 4.5){C4}
\pnode(6.04, 4.5){C5}
\pnode(7.4, 3.56){C7}
\pnode(5.25, 1.83){C8}
\pnode(2.6, 2.05){C9}

\uput[dl](B05){$B_{0.5}$}
\uput[ul](B15){$B_{1.5}$}
\uput[u](P){$P$}
\uput[u](Q){$Q$}
\uput[ur](B65){$B_{6.5}$}
\uput[dr](B75){$B_{7.5}$}
\uput[d](B85){$B_{8.5}$}

\uput[r](C1){$C_1$}
\uput[d](C3){$C_3$}
\uput[d](C4){$C_4$}
\uput[d](C5){$C_5$}
\uput[l](C7){$C_7$}
\uput[ul](C8){$C_8$}
\uput[ur](C9){$C_9$}

\pspolygon[showpoints=true](B05)(B15)(P)(Q)(B65)(B75)(B85)
\ncline[linestyle=dashed]{B05}{P}
\ncline[linestyle=dashed]{Q}{B75}
\ncline[linestyle=dashed]{B65}{B85}
\ncline[linestyle=dashed]{B75}{B05}
\ncline[linestyle=dashed]{B85}{B15}
\pspolygon[showpoints=true](C1)(P)(C3)(C4)(C5)(Q)(C7)(C8)(C9)

\end{pspicture}
\caption{The construction of $C=T^2(A)$ from $B=T(A)$ for $A \in X_{\{3,5\}}$.  Here, $P=B_{2.5}=B_{3.5}=C_2$ and $Q=B_{4.5}=B_{5.5}=C_6$.}
\label{figBC2}
\end{figure}

\begin{figure}
\begin{pspicture}(9,5)
\pnode(0.66, 2.32){C1}
\pnode(1.28, 4.2){C2}
\pnode(2.56, 4.2){R}
\pnode(3.94, 4.2){C4}
\pnode(5.42, 4.2){S}
\pnode(7.92, 4.2){C6}
\pnode(7.16, 1.78){C7}
\pnode(4.78, 0.62){C8}
\pnode(2.34, 0.7){C9}

\pnode(2.02, 1.76){D05}
\pnode(1.58, 3.23){D15}
\pnode(6.55, 2.63){D65}
\pnode(5.46, 1.4){D75}
\pnode(3.8, 1.03){D85}

\uput[l](C1){$C_1$}
\uput[ul](C2){$C_2$}
\uput[u](R){$R$}
\uput[u](C4){$C_4$}
\uput[u](S){$S$}
\uput[ur](C6){$C_6$}
\uput[dr](C7){$C_7$}
\uput[d](C8){$C_8$}
\uput[d](C9){$C_9$}

\uput[ur](D05){$D_{0.5}$}
\uput[r](D15){$D_{1.5}$}
\uput[l](D65){$D_{6.5}$}
\uput[ul](D75){$D_{7.5}$}
\uput[u](D85){$D_{8.5}$}

\pspolygon[showpoints=true](C1)(C2)(R)(C4)(S)(C6)(C7)(C8)(C9)
\ncline[linestyle=dashed]{C1}{R}
\ncline[linestyle=dashed]{S}{C7}
\ncline[linestyle=dashed]{C6}{C8}
\ncline[linestyle=dashed]{C7}{C9}
\ncline[linestyle=dashed]{C8}{C1}
\ncline[linestyle=dashed]{C9}{C2}
\pspolygon[showpoints=true](D05)(D15)(R)(S)(D65)(D75)(D85)
\end{pspicture}
\caption{The construction of $D=T^3(A)$ from $C=T^2(A)$ for $A \in X_{\{3,5\}}$.  Here, $R=C_3=D_{2.5}=D_{4.5}$ and $S=C_5=D_{3.5}=D_{5.5}$.}
\label{figCD2}
\end{figure}

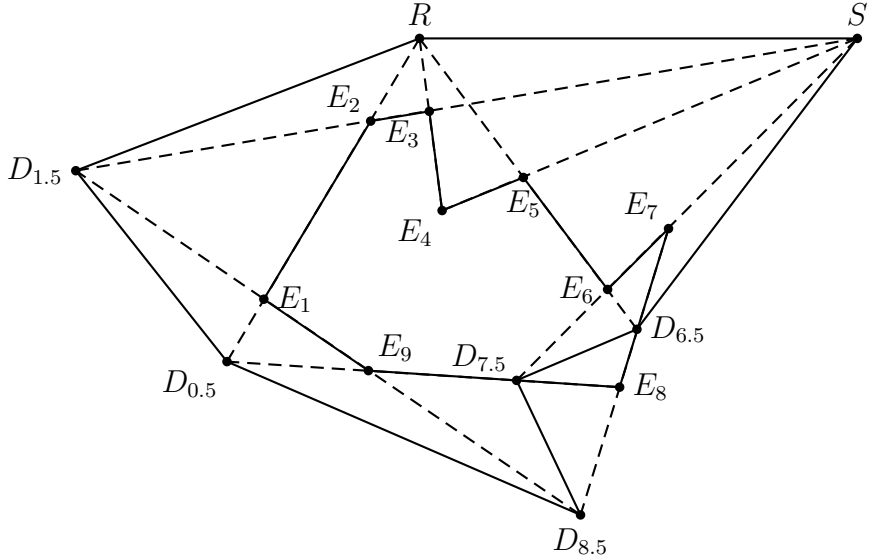
\begin{figure}
\begin{pspicture}(11,8)
\pnode(2.56, 2.7){D05}
\pnode(0.55, 5.24){D15}
\pnode(5.12,7){R}
\pnode(10.94,7){S}
\pnode(8.01, 3.13){D65}
\pnode(6.41, 2.45){D75}
\pnode(7.26, 0.66){D85}

\pnode(3.05, 3.53){E1}
\pnode(4.47, 5.9){E2}
\pnode(5.25, 6.03){E3}
\pnode(5.42, 4.71){E4}
\pnode(6.5, 5.15){E5}
\pnode(7.62, 3.66){E6}
\pnode(8.43, 4.47){E7}
\pnode(7.78, 2.36){E8}
\pnode(4.44, 2.58){E9}

\uput[dl](D05){$D_{0.5}$}
\uput[l](D15){$D_{1.5}$}
\uput[u](R){$R$}
\uput[u](S){$S$}
\uput[r](D65){$D_{6.5}$}
\uput[ul](D75){$D_{7.5}$}
\uput[d](D85){$D_{8.5}$}

\uput[r](E1){$E_1$}
\uput[ul](E2){$E_2$}
\uput[dl](E3){$E_3$}
\uput[dl](E4){$E_4$}
\uput[d](E5){$E_5$}
\uput[l](E6){$E_6$}
\uput[ul](E7){$E_7$}
\uput[r](E8){$E_8$}
\uput[ur](E9){$E_9$}

\pspolygon[showpoints=true](D05)(D15)(R)(S)(D65)(D75)(D85)
\ncline[linestyle=dashed]{D85}{D15}
\ncline[linestyle=dashed]{D05}{R}
\ncline[linestyle=dashed]{D15}{S}
\ncline[linestyle=dashed]{R}{E4}
\ncline[linestyle=dashed]{S}{E4}
\ncline[linestyle=dashed]{R}{D65}
\ncline[linestyle=dashed]{S}{D75}
\ncline[linestyle=dashed]{E7}{D85}
\ncline[linestyle=dashed]{E8}{D05}
\pspolygon[showpoints=true](E1)(E2)(E3)(E4)(E5)(E6)(E7)(E8)(E9)
\end{pspicture}
\caption{The construction of $E=T^4(A)$ from $D=T^3(A)$ for $A \in X_{\{3,5\}}$.  Here, $R=D_{2.5}=D_{4.5}$ and $S=D_{3.5}=D_{5.5}$.}
\label{figDE2}
\end{figure}

As before, $B$ can be constructed using the standard pentagram map.  In this case, it will have three sides equal to $l$, namely $b_2$, $b_4$, and $b_6$.  As a result, the six consecutive vertices $B_{1.5},B_{2.5},\ldots,B_{6.5}$ will all lie on $l$.  Moreover, we have
\begin{align*}
B_{2.5} &= B_{3.5} = \meet{l}{b_3} \\
B_{4.5} &= B_{5.5} = \meet{l}{b_5}
\end{align*}

The sides $c_j = \join{B_{j-1}}{B_{j+1}}$ of $C$ are all defined, although $c_{2.5},c_{3.5},c_{4.5},c_{5.5}$ all equal $l$.  So the first problems arise in constructing the vertices $C_3$, $C_4$, and $C_5$.  In general, $C_3$ only depends on vertices $A_0$ through $A_6$ of the original polygon, so the assumption $A_7 \in l$ is irrelevant for its construction.  This puts us in the context of the previous subsection, so as there we have
\begin{displaymath}
C_3 = \textrm{\TripleConjugate}(B_{1.5},B_{2.5},A_3,B_{4.5},B_{3.5})
\end{displaymath}
By symmetry there is a similar construction of $C_5$:
\begin{displaymath}
C_5 = \textrm{\TripleConjugate}(B_{3.5},B_{4.5},A_5,B_{6.5},B_{5.5})
\end{displaymath}

The situation with $C_4$ is more complicated.  Ordinarily, we would use the fact 
\begin{displaymath}
[B_{2.5},B_{3.5},A_4,B_{5.5},B_{4.5},C_4] = -1
\end{displaymath}
and solve for $C_4$.  However, $B_{2.5}=B_{3.5}$ and $B_{5.5}=B_{4.5}$ so the triple ratio comes out to $0/0$.  As such, we can not use this method to construct $C_4$.  In fact, it turns out that $C_4$ is simply not well-defined.  

The fact that an intermediate vertex is not well-defined causes great difficulty in the current approach to devising straightedge constructions.  In the following sections, we demonstrate how enriching the input $A$ with first-order data counteracts this difficulty and leads to a general algorithm.  Before moving on, we finish describing a construction particular to the present context which works around the matter of $C_4$.

Recall that sides $c_{2.5}$ through $c_{5.5}$ of $C$ all equal $l$, so its vertices $C_2$ through $C_6$ all lie on this line.  As such, we know $d_3 = \join{C_2}{C_4}=l$ and $d_5 = \join{C_4}{C_6} = l$ even though $C_4$ is itself not well-defined.  The rest of the sides of $D$ are constructed similarly, and of note $d_4$ also equals $l$.  The other sides are all generic so we can construct $D_j = \meet{d_{j-\frac{1}{2}}}{d_{j+\frac{1}{2}}}$ for all $j$ other than $3.5$ and $4.5$.  For these two vertices, we work backwards.  We know ultimately that $e_{1.5},e_{3.5},e_{5.5}$ will be concurrent.  But $\meet{e_{1.5}}{e_{3.5}} = D_{2.5}$ and  $\meet{e_{3.5}}{e_{5.5}} = D_{4.5}$ so $D_{4.5} = D_{2.5}$.  Similarly, the fact that $e_{2.5},e_{4.5},e_{6.5}$ are concurrent implies that $D_{3.5} = D_{5.5}$.

For the final step, let $e_j = \join{D_{j-1}}{D_{j+1}}$ for all $j$ besides $3.5$ and $4.5$.  The usual construction fails for $e_{3.5}$ because $D_{2.5}=D_{4.5}$ and the construction involving triple conjugates also fails because $d_3=d_4=d_5$.  However, $e_{3.5}$ certainly is well-defined as it is a side of $E = T^4(A) \in \mathcal{P}_n$.  Through trial and error we discovered the following construction for $e_{3.5}$, and by symmetry, one for $e_{4.5}$.

\begin{prop}
\label{prope35}
Under the assumptions of this subsection
\begin{align*}
e_{3.5} &= \join{(\meet{\join{(\meet{b_5}{d_2})}{C_5}}{c_{6.5}})}{C_3} \\
e_{4.5} &= \join{(\meet{\join{(\meet{b_3}{d_6})}{C_3}}{c_{1.5}})}{C_5}
\end{align*}
\end{prop}

\begin{rem}
\label{remNoProof}
In principle, results like Proposition \ref{prope35} can be proven computationally.  In instances for which we are unaware of a more illuminating proof, we will simply fall back on this sort of reasoning.  The computations required are generally quite tedious, so we will tend to omit them.
\end{rem}

\section{Decorated polygons}
Let $A$ be a twisted polygon which is a singular point of $T^k$.  As explained in the previous section, we can attempt to define $T^k(A)$ as a limit of $T^k(A(t))$ where $A(t)$ is a curve in the space of polygons passing through $A=A(0)$.  As we saw in Section \ref{secStraightEdge2}, the result sometimes depends on the choice of the curve.  This suggests a different approach to constructing the first nonsingular iterate $T^m(A)$.  Start by fixing arbitrarily the one-parameter family $A(t)$.  With respect to this choice the intermediate polygons $T^k(A)$ are well-defined.  Constructing them in turn we eventually get $T^m(A)$.  Since $A$ is not a singular point of $T^m$, the final result will not depend on the choice of $A(t)$.

Working with actual curves would be difficult.  However, all that will actually matter will be the first order behavior of the curve near $t=0$.  This information can be encoded using geometric data which we call decorations.

Let $A$ be a point in the projective plane, and let $\gamma$ be a smooth curve with $\gamma(0) = A$.  Define the associated \emph{decoration} of $A$, denoted $A^*$, to be the tangent line of $\gamma$ at $A$:
\begin{displaymath}
A^* = \lim_{t \to 0} \join{A}{\gamma(t)}
\end{displaymath}
When defined, $A^*$ is a line passing through $A$.

By the same token, if $a$ is a line in the projective plane then $a$ can be thought of as a point in the dual plane.  Given a curve $\gamma$ through that point we can define the \emph{decoration} $a^*$ as
\begin{displaymath}
a^* = \lim_{t \to 0} \meet{a}{\gamma(t)}
\end{displaymath}
When defined, $a^*$ is a point lying on $a$.

Finally, let $A$ be a twisted polygon and $\gamma$ a curve in the space of twisted polygons with $\gamma(0)=A$.  Then $\gamma$ determines a curve in the plane through each vertex of $A$ and a curve in the dual plane through each side of $A$.  By the above, we can define decorations on each of these individual objects.

\begin{defin}
\label{defDecPoly}
A \emph{decorated polygon} is a twisted polygon $A$ together with the decorations of each of its vertices and sides induced by some curve $\gamma$ in the space of twisted polygons with $\gamma(0)=A$.
\end{defin}

Decorated polygons will be denoted by the appropriate script letter.  For instance if the underlying polygon is $A$ then the decorated polygon will be called $\mathcal{A}$.  It is possible for different curves to give rise to the same decorated polygon $\mathcal{A}$.  As such, $\mathcal{A}$ corresponds to a whole class of curves.  We will call any curve $\gamma$ in this class a \emph{representative} of $\mathcal{A}$.

\begin{rem}
Definition \ref{defDecPoly} refers to the space of twisted polygons.  This should not be confused with $\mathcal{P}_n$ or $\mathcal{P}_n^*$ which are spaces of projective equivalence classes of twisted polygons.  In general, we will be working with actual polygons instead of equivalence classes of polygons for the remainder of the paper.
\end{rem}

Given a collection of geometric objects which satisfy certain incidences (e.g. the vertices and sides of a polygon), a consistent choice of decorations of these objects may have to satisfy some relations.  The simplest example of such a relation occurs for a closed triangle.

\begin{prop}
\label{decRel}
In $\triangle ABC$ let $a = \join{B}{C}$, $b = \join{A}{C}$, and $c = \join{A}{B}$ (see Figure \ref{figDecTri}).  Let $A(t)$, $B(t)$, and $C(t)$ be curves through the vertices and use them to define curves corresponding to the sides (e.g. $a(t) = \join{B(t)}{C(t)}$).  Then the corresponding decorations satisfy
\begin{equation} \label{triangle}
[A,c^*,B,a^*,C,b^*] = [a,C^*,b,A^*,c,B^*]^{-1}
\end{equation}
\end{prop}

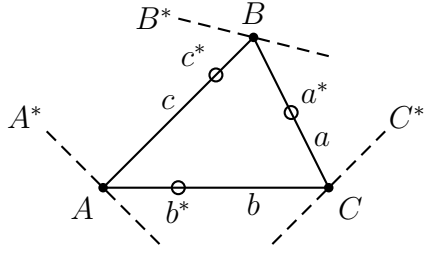
\begin{figure}
\begin{pspicture}(5,4)
\pspolygon[showpoints=true](1,1)(3,3)(4,1)
\uput[dl](1,1){$A$}
\uput[u](3,3){$B$}
\uput[dr](4,1){$C$}

\uput{2pt}[ur](3.75,1.5){$a$}
\uput{2pt}[d](3,1){$b$}
\uput{2pt}[ul](2,2){$c$}

\pscircle(3.5,2){.1}
\uput[ur](3.5,2){$a^*$}
\pscircle(2,1){.1}
\uput[d](2,1){$b^*$}
\pscircle(2.5,2.5){.1}
\uput[ul](2.5,2.5){$c^*$}

\psline[linestyle=dashed](.25,1.75)(1.75,.25)
\uput{2pt}[ul](.25,1.75){$A^*$}
\psline[linestyle=dashed](2,3.25)(4,2.75)
\uput{2pt}[l](2,3.25){$B^*$}
\psline[linestyle=dashed](3.25,.25)(4.75,1.75)
\uput{2pt}[ur](4.75,1.75){$C^*$}
\end{pspicture}
\caption{A decorated triangle}
\label{figDecTri}
\end{figure}

\begin{lem}
\label{lemTwoTri}
Let $\triangle A'B'C'$ be another triangle with $a' = \join{B'}{C'}$, $b' = \join{A'}{C'}$, and $c' = \join{A'}{B'}$ (see Figure \ref{figTwoTri}).  Then
\begin{equation} \label{twoTriangles}
[A,\meet{c}{c'},B,\meet{a}{a'},C,\meet{b}{b'}] = [a',\join{C}{C'},b',\join{A}{A'},c',\join{B}{B'}]^{-1}
\end{equation}
\end{lem}
\begin{proof}
See Remark \ref{remNoProof}.
\end{proof}

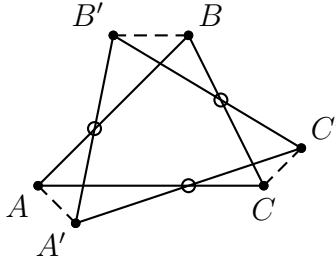
\begin{figure}
\begin{pspicture}(5,4)

\pnode(1,1){A}
\pnode(3,3){B}
\pnode(4,1){C}
\pnode(1.5,.5){A'}
\pnode(2,3){B'}
\pnode(4.5,1.5){C'}

\uput[dl](A){$A$}
\uput[ur](B){$B$}
\uput[d](C){$C$}
\uput[dl](A'){$A'$}
\uput[ul](B'){$B'$}
\uput[ur](C'){$C'$}

\ncline[arrows=*-*]{A}{B}
\ncline[arrows=*-*]{B}{C}
\ncline[arrows=*-*]{C}{A}
\ncline[arrows=*-*]{A'}{B'}
\ncline[arrows=*-*]{B'}{C'}
\ncline[arrows=*-*]{C'}{A'}

\ncline[linestyle=dashed]{A}{A'}
\ncline[linestyle=dashed]{B}{B'}
\ncline[linestyle=dashed]{C}{C'}
\pscircle(1.75,1.76){.1}
\pscircle(3.43,2.14){.1}
\pscircle(3,1){.1}
\end{pspicture}
\caption{The points and lines involved in Lemma \ref{lemTwoTri}}
\label{figTwoTri}
\end{figure}

\begin{proof}[Proof of Proposition \ref{decRel}]
In Lemma \ref{lemTwoTri}, take $A' = A(t)$, $B' = B(t)$, and $C'=C(t)$ and consider the limit as $t$ goes to 0.  In the limit, the vertices and sides of $\triangle A'B'C'$ approach their counterparts in $\triangle ABC$.  By definition, $\join{A}{A'}$ approaches $A^*$ and similarly for the other vertices and sides.  So the limit of \eqref{twoTriangles} is precisely \eqref{triangle}.
\end{proof}

\begin{rem}
For each $n$, there is a relation similar to \eqref{triangle} among the decorations of a closed $n$-gon.  Moreover, this is the only relation that holds.  Hence, one can pick $n$ vertex decorations and $n-1$ sides decorations (or the other way around) independently, and the last decoration is then determined.  The space of decorations of a fixed polygon $A$, then, has dimension $2n-1$.  One can check that this space naturally corresponds to the projectivized tangent space of $A$ within the space of closed $n$-gons.  A similar statement is probably true for twisted polygons, but we do not understand what the relations are among the individual decorations.
\end{rem}

The next proposition shows that repeated applications of \eqref{triangle} can be used to extend the pentagram map to decorated polygons. 

\begin{prop}
\label{propTTildeWorks}
Let $A(t)$ be a curve in the space of polygons and let $B(t)=T(A(t))$ for all $t$.  Let $\mathcal{A}$ and $\mathcal{B}$ be the corresponding decorations of $A=A(0)$ and $B=B(0)$ respectively.  Then $\mathcal{B}$ is uniquely determined by $\mathcal{A}$.
\end{prop}

\begin{proof}
Given an index $i$ of $A$, consider the triangle with vertices $A_{i-1}$, $A_i$, and $A_{i+1}$ (see Figure \ref{figConSide}).  It has all three vertices and two of its sides coming from $A$.  The last side is $\join{A_{i-1}}{A_{i+1}}=b_i$.  Applying Proposition \ref{decRel} to this triangle, then, expresses $b_i^*$ in terms of $\mathcal{A}$.

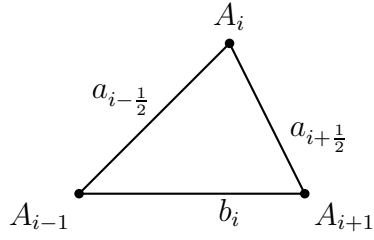
\begin{figure}
\begin{pspicture}(5,4)
\pspolygon[showpoints=true](1,1)(3,3)(4,1)
\uput[dl](1,1){$A_{i-1}$}
\uput[u](3,3){$A_i$}
\uput[dr](4,1){$A_{i+1}$}

\uput{2pt}[ur](3.75,1.5){$a_{i+\frac{1}{2}}$}
\uput{2pt}[d](3,1){$b_i$}
\uput{2pt}[ul](2,2){$a_{i-\frac{1}{2}}$}
\end{pspicture}
\caption{The triangle used to compute $b_i^*$}
\label{figConSide}
\end{figure}

Next, consider the triangle with vertices $A_{j-\frac{1}{2}}$, $B_j$ and $A_{j+\frac{1}{2}}$ for some index $j$ of $B$ (see Figure \ref{figConVert}).  Two of its vertices and one of its sides belong to $A$.  The other two sides are in fact sides of $B$, namely
\begin{align*}
\join{A_{j-\frac{1}{2}}}{B_j} &= b_{j+\frac{1}{2}} \\
\join{A_{j+\frac{1}{2}}}{B_j} &= b_{j-\frac{1}{2}} 
\end{align*}
These were both decorated in the previous step.  Another application of Proposition \ref{decRel}, then, determines $B_j^*$.

\begin{figure}
\begin{pspicture}(5,4)
\pspolygon[showpoints=true](1,3)(3,1)(4,3)
\uput[ul](1,3){$A_{j-\frac{1}{2}}$}
\uput[d](3,1){$B_j$}
\uput[ur](4,3){$A_{j+\frac{1}{2}}$}

\uput{2pt}[dr](3.75,2.5){$b_{j-\frac{1}{2}}$}
\uput{2pt}[u](3,3){$a_j$}
\uput{2pt}[dl](2,2){$b_{j+\frac{1}{2}}$}
\end{pspicture}
\caption{The triangle used to compute $B_j^*$}
\label{figConVert}
\end{figure}
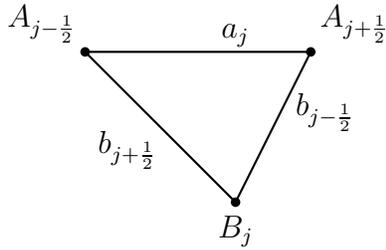
\end{proof}

The procedure above to construct $\mathcal{B}$ from $\mathcal{A}$ should be thought of as a lift of the pentagram map to the space of decorated polygons.  To distinguish this operation from the original map, write $\mathcal{B} = \tilde{T}(\mathcal{A})$.  The construction defining $\tilde{T}$ is given in Algorithm \ref{algTLift}.

\begin{algorithm}
\caption{$\tilde{T}(\mathcal{A})$}
\begin{algorithmic}
\FORALL{$i$}
\STATE $b_i := \join{A_{i-1}}{A_{i+1}}$  
\STATE $b_i^*$ := \DecorateSide($\mathcal{A}$, $b_i$, $i$)
\ENDFOR
\FORALL{$j$}
\STATE $B_j := \meet{b_{j-\frac{1}{2}}}{b_{j+\frac{1}{2}}}$ 
\STATE $B_j^*$ := \DecorateVertex($\mathcal{A}$, $(b_i)$, $(b_i^*)$, $B_j$, $j$)
\ENDFOR
\RETURN $\mathcal{B}$
\end{algorithmic}
\label{algTLift}
\end{algorithm}

The subroutines \DecorateSide\ and \DecorateVertex\ build the triangles in Figures \ref{figConSide} and \ref{figConVert} respectively, and use \eqref{triangle} to compute the desired decoration. 

\begin{rem}
We will only be using decorated polygons and the map $\tilde{T}$ as tools in our straightedge constructions.  However, these are likely interesting objects to study in their own right.  Some immediate questions come to mind such as
\begin{itemize}
\item What would be a good set of coordinates on the space of decorated polygons?
\item In such coordinates, does the map $\tilde{T}$ take a nice form?
\item Does $\tilde{T}$ define a discrete integrable system?
\end{itemize}
\end{rem}

\section{Degenerations}
We saw in the previous section that \eqref{triangle} is the only identity needed to apply the pentagram map to a generic decorated polygon.  However, the motivation for introducing decorations is to handle degenerate cases.  In this section, we introduce methods which will eventually be used to apply the pentagram map to a large class of degenerate polygons.  Everything will be expressed in terms of local rules involving triangles and complete quadrilaterals.

\subsection{Triangles}
Let $A(t)$, $B(t)$, $C(t)$ be curves in the plane passing through points $A$, $B$, and $C$ at time $t=0$.  Assume that for all $t \neq 0$, the points $A(t)$, $B(t)$, and $C(t)$ are in general position, and let $a(t)$, $b(t)$, and $c(t)$ denote the sides of the triangle they form.  We allow $A$, $B$, and $C$ to be collinear, or even equal to each other.  However, assume that the limits
\begin{align*}
a &= \lim_{t \to 0} a(t) \\
b &= \lim_{t \to 0} b(t) \\
c &= \lim_{t \to 0} c(t) 
\end{align*}
all exist.  Assume that the limits defining the decorations $A^*,B^*,C^*,a^*,b^*,c^*$ all exist as well.

Now, Proposition \ref{decRel} did not allow for $A$, $B$ and $C$ to be collinear.  However, by continuity \eqref{triangle} still holds in the present context.

\begin{prop}
\label{propThirdVert}
Assume that the decorations of the triangle are generic (i.e. distinct from each other and from the vertices and sides of the triangle.)  Then the vertex $B$ is uniquely determined by the sides $a,b,c$, the vertices $A,C$, and all the corresponding decorations.  
\end{prop}

\begin{proof}
If $a \neq c$, then $B = \meet{a}{c}$, so assume $a = c$.  There are two cases depending on if $b$ equals the other sides.

If $b \neq a=c$ then $A = \meet{b}{c} = \meet{b}{a} = C$.  In general, if $P_1=P_5$ then 
\begin{align*}
[P_1,P_2,P_3,P_4,P_5,P_6] &= \frac{P_1P_2}{P_2P_3}\frac{P_3P_4}{P_4P_5}\frac{P_5P_6}{P_6P_1} \\
                          &= -\frac{P_1P_2}{P_2P_3}\frac{P_3P_4}{P_4P_1} \\
                          &= -[P_1,P_2,P_3,P_4] 
\end{align*}
so \eqref{triangle} simplifies to 
\begin{equation}
\label{degTriangle}
[A,c^*,B,a^*] = [a,C^*,b,A^*]^{-1}
\end{equation}
Note that the dependence on $B^*$ has disappeared, so this identity determines $B$ from the given geometric data.

Alternately, suppose $b=a=c$.  Whenever $P_1=P_3=P_5$ we have
\begin{align*}
[P_1,P_2,P_3,P_4,P_5,P_6] &= \frac{P_1P_2}{P_2P_3}\frac{P_3P_4}{P_4P_5}\frac{P_5P_6}{P_6P_1} \\
                          &= \frac{P_1P_2}{P_2P_1}\frac{P_3P_4}{P_4P_3}\frac{P_5P_6}{P_6P_5} \\
                          &= -1 
\end{align*}
Consequently, \eqref{triangle} becomes
\begin{displaymath}
[A,c^*,B,a^*,C,b^*] = -1
\end{displaymath}
so again we can construct $B$.
\end{proof}

To sum up, if $a=c$ in a triangle then \eqref{triangle} can be used to construct the vertex $B$.  The downside is that this identity can no longer be used to determine the decoration $B^*$.  In fact, $B^*$ is independent from the rest of the triangle when $a=c$.  As such, we will need more data to construct vertex decorations when degeneracies occur.

\subsection{Complete quadrilaterals}
A \emph{complete quadrilateral} is a projective configuration consisting of four lines (called sides) in general position together with the six points (called vertices) at which they intersect.  Call the sides $l_1,l_2,l_3,l_4$ and call the vertices $A,B,C,D,E,F$ as in Figure \ref{figQuad}.

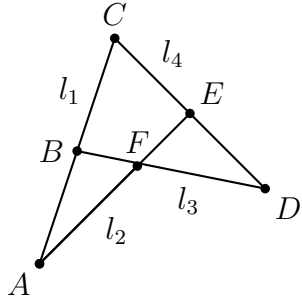
\begin{figure}
\begin{pspicture}(6,5)
\psline[showpoints=true](3,3)(1,1)(2,4)(4,2)(1.5,2.5)
\psline[showpoints=true](1,1)(2.3,2.3)
\uput[dl](1,1){$A$}
\uput[l](1.5,2.5){$B$}
\uput[u](2,4){$C$}
\uput[dr](4,2){$D$}
\uput[ur](3,3){$E$}
\uput[u](2.3,2.3){$F$}
\uput[ul](1.67,3){$l_1$}
\uput[dr](1.75,1.75){$l_2$}
\uput[dl](3.25,2.15){$l_3$}
\uput[ur](2.5,3.5){$l_4$}
\end{pspicture}
\caption{A complete quadrilateral}
\label{figQuad}
\end{figure}

As with triangles, we define degenerate complete quadrilaterals to be configurations that can be obtained as a limit of ordinary complete quadrilaterals.  More precisely, let $l_1(t), \ldots, l_4(t)$ and $A(t), \ldots, F(t)$ be smooth curves which define a complete quadrilateral at each time $t \neq 0$.  Let $l_1=l_1(0), l_2=l_2(0), \ldots, F=F(0)$.  Then some or all of the sides $l_i$ could be equal.  Assume that the decorations $l_1^*, l_2^*, \ldots, F^*$ are all defined.

\begin{prop}
\label{propMultiConj}
Consider a degenerate complete quadrilateral with $l_1=l_2=l_3=l_4$.  Then the vertex decoration $F^*$ is uniquely determined by the other vertices, sides, and decorations. 
\end{prop}

\begin{proof}
By Menelaus' theorem, $[A,B,C,D,E,F] = -1$.  According to the appendix, $F$ can be constructed from the other vertices using a construction as in Figure \ref{figTripConj}.  The idea of the present construction is to build a configuration as in that figure at each time $t$.

Choose generically in the plane a point $c^*$ and lines $P^*$ and $(C')^*$ (the reason for these names will be clear shortly.)  For each $t$, define
\begin{align*}
c(t) &= \join{C(t)}{c^*} \\
P(t) &= \meet{c(t)}{P^*} \\
C'(t) &= \meet{c(t)}{(C')^*}
\end{align*}
Note that
\begin{displaymath}
\lim_{t \to 0}\meet{c(t)}{c(0)} = \lim_{t \to 0}\meet{\join{C(t)}{c^*}}{\join{C}{c^*}} = c^*
\end{displaymath}
which justifies the notation $c^*$.  Similar remarks hold for $P^*$ and $C'^*$.  Defining 
\begin{align*}
B'(t) &= \meet{\join{A(t)}{C'(t)}}{\join{B(t)}{P(t)}} \\
D'(t) &= \meet{\join{E(t)}{C'(t)}}{\join{D(t)}{P(t)}} 
\end{align*}
we get at time $t=0$ a configuration as in Figure \ref{figTripConj}.  In particular $B'$, $D'$, and $F$ are collinear.  The proof of this fact given in the appendix generalizes to show that $B'(t)$, $D'(t)$, and $F(t)$ are collinear for all $t$.  Let $l_3'(t)$ denote the common line.  By way of notation, let $l_1'(t) = \join{A(t)}{C'(t)}$, $l_4'(t) = \join{E(t)}{C'(t)}$, $b(t) = \join{B(t)}{P(t)}$, and $d(t) = \join{D(t)}{P(t)}$.  

We are given decorations of $A,B, \ldots, E$ and $l_1,l_2,l_3,l_4$ from the outset.  We chose arbitrarily decorations of $P$, $C'$ and $c$.  In a generic triangle, knowing five of the six decorations determines the sixth by Proposition \ref{decRel}.  This fact can be used to find all missing decorations in our configuration.  To start, use the first triangle in Figure \ref{figDecTripConj1} to find $(l_1')^*$ and the second triangle to find $b^*$.  Once these decorations are found, the third triangle in the figure can be used to determine $(B')^*$.  A similar method is used to find $(l_4')^*$, $d^*$, and then $(D')^*$.  Finally, use the first triangle in Figure \ref{figDecTripConj2} to find $(l_3')^*$ and the other triangle to find $F^*$.
\end{proof}

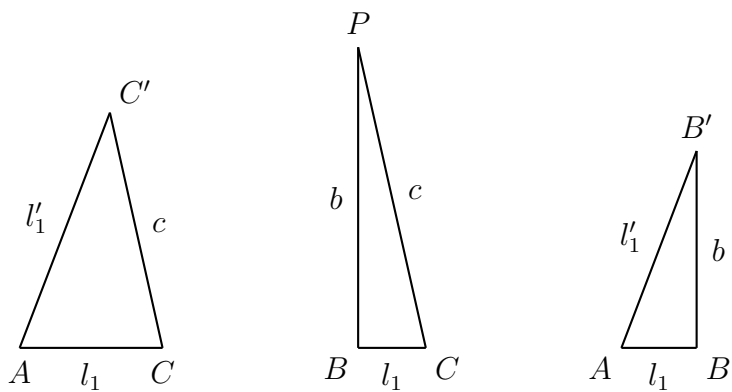
\begin{figure}
\begin{pspicture}(11,6)
\rput(0,0){
\pnode(1,1){A}
\pnode(2.9,1){C}
\pnode(2.2,4.13){C1}

\uput[d](A){$A$}
\uput[d](C){$C$}
\uput[ur](C1){$C'$}
\ncline{A}{C1}
\naput{$l_1'$}
\ncline{C1}{C}
\naput{$c$}
\ncline{C}{A}
\naput{$l_1$}
}

\rput(3.5,0){
\pnode(2,1){B}
\pnode(2,5){P}
\pnode(2.9,1){C}

\uput[dl](B){$B$}
\uput[u](P){$P$}
\uput[dr](C){$C$}
\ncline{B}{P}
\naput{$b$}
\ncline{P}{C}
\naput{$c$}
\ncline{C}{B}
\naput{$l_1$}
}

\rput(8,0){
\pnode(1,1){A}
\pnode(2,3.62){B1}
\pnode(2,1){B}

\uput[dl](A){$A$}
\uput[u](B1){$B'$}
\uput[dr](B){$B$}
\ncline{A}{B1}
\naput{$l_1'$}
\ncline{B1}{B}
\naput{$b$}
\ncline{B}{A}
\naput{$l_1$}
}

\end{pspicture}
\caption{The triangles used to find $(l_1')^*$, $b^*$, and finally $(B')^*$}
\label{figDecTripConj1}
\end{figure}

\begin{figure}
\begin{pspicture}(6,4)
\rput(-3,-2){
\pnode(2,3.62){B1}
\pnode(2.2,4.14){C1}
\pnode(3.42,2.89){D1}

\uput[dl](B1){$B'$}
\uput[u](C1){$C'$}
\uput[dr](D1){$D'$}
\ncline{B1}{C1}
\naput{$l_1'$}
\ncline{C1}{D1}
\naput{$l_4'$}
\ncline{D1}{B1}
\naput{$l_3'$}
}

\rput(.5,0){
\pnode(3.42,2.89){D1}
\pnode(7.14,1){F}
\pnode(5.3,1){E}

\uput[ul](D1){$D'$}
\uput[dr](F){$F$}
\uput[dl](E){$E$}
\ncline{D1}{F}
\naput{$l_3'$}
\ncline{F}{E}
\naput{$l_2$}
\ncline{E}{D1}
\naput{$l_4'$}
}

\end{pspicture}
\caption{The triangles used to find $(l_3')^*$ and $F^*$}
\label{figDecTripConj2}
\end{figure}
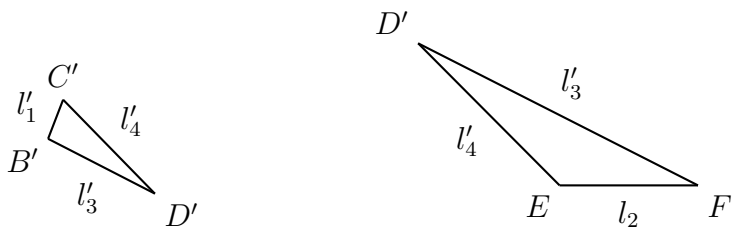

\subsection{Degenerate polygons}

A \emph{degenerate polygon} is a sequences of points and an interlacing sequence of lines, which occur as the limits of the vertices and sides, respectively, of some twisted polygons.  More precisely, if $A(t)$ is a twisted polygon for $t \neq 0$ and the appropriate limits are defined, then we get a degenerate polygon $A$ with vertices $A_i = \lim_{t\to0} A_i(t)$ and sides $a_j = \lim_{t\to0}a_j(t)$.  Fixing such a curve $A(t)$, we get decorations for $A$ as before.

Using our understanding of degenerate triangles and complete quadrilaterals, we are ready to state a version of $\tilde{T}$ which works for degenerate polygons.  As will be explained, the algorithm requires two consecutive iterates of the pentagram map, $\mathcal{A}$ and $\mathcal{B}$, as input.  The output is the iterate $\mathcal{C}$ which follows these two.  The new procedure is called $\tilde{T}_2$ and is given in Algorithm \ref{algTLift2}.  In the algorithm, $j$ ranges over the index set of $B$ and $i$ ranges over the index set of $C$.  A description of the subroutines appearing in the algorithm will follow.

\begin{algorithm}
\caption{$\tilde{T}_2$($\mathcal{A}$, $\mathcal{B}$)}
\begin{algorithmic}
\FORALL{$j$}
\STATE $c_j$ := \ConstructSideB($\mathcal{B}$, $j$)
\ENDFOR
\FORALL{$j$}
\STATE $c_j^*$ := \DecorateSideB($\mathcal{A}$, $\mathcal{B}$, $c_j$, $j$)
\ENDFOR
\FORALL{$i$}
\STATE $C_i$ := \ConstructVertexB($\mathcal{B}$, $(c_j)$, $(c_j^*)$, $i$) 
\ENDFOR
\FORALL{$i$}
\STATE $C_i^*$ := \DecorateVertexB($\mathcal{A}$, $\mathcal{B}$, $(c_j)$, $(c_j^*)$, $C_i$, $i$) 
\ENDFOR
\RETURN $\mathcal{C}$
\end{algorithmic}
\label{algTLift2}
\end{algorithm}

The outline of this algorithm is the same as that of the original version of $\tilde{T}$.  First the sides of $\mathcal{C}$ are constructed, then the side decorations, then the vertices, and finally the vertex decorations.  Each individual step, though, is made more complicated by the possibility of degeneracies.  

The subroutine \ConstructVertexB\ works with the triangle pictured in Figure \ref{figConVert2}.  All of the components of the triangle besides $C_i$ have already been constructed.  Hence by Proposition \ref{propThirdVert}, it is always possible to construct $C_i$.

\begin{figure}
\begin{pspicture}(5,4)
\pspolygon[showpoints=true](1,3)(3,1)(4,3)
\uput[ul](1,3){$B_{i-\frac{1}{2}}$}
\uput[d](3,1){$C_i$}
\uput[ur](4,3){$B_{i+\frac{1}{2}}$}

\uput{2pt}[dr](3.75,2.5){$c_{i-\frac{1}{2}}$}
\uput{2pt}[u](3,3){$b_i$}
\uput{2pt}[dl](2,2){$c_{i+\frac{1}{2}}$}
\end{pspicture}
\caption{The triangle used in \ConstructVertexB}
\label{figConVert2}
\end{figure}
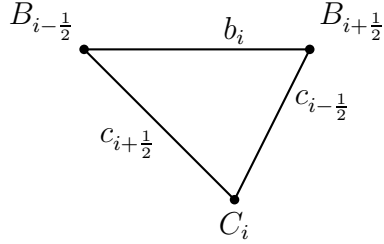

The subroutine \DecorateVertexB\ begins from this same triangle.  If $c_{i-\frac{1}{2}} \neq c_{i+\frac{1}{2}}$ then \eqref{triangle} is used to compute $C_i^*$.  Otherwise, more data is needed.  Consider the complete quadrilateral in Figure \ref{figDecVert}.  We are assuming that two of its sides, namely $c_{i-\frac{1}{2}}$ and $c_{i+\frac{1}{2}}$, are equal.  This forces the five vertices other than $A_i$ to lie on the common line.  Generically, these five vertices are distinct forcing the two remaining sides (namely $b_{i-1}$ and $b_{i+1}$) of the complete quadrilateral to be equal to each other and to $c_{i-\frac{1}{2}}=c_{i+\frac{1}{2}}$.  This puts us in the situation of Proposition \ref{propMultiConj}.  All vertices besides $C_i$ and all sides have been decorated already, so the procedure can determine the decoration on $C_i$. 

We used above the fact that if two sides of a complete quadrilateral are equal, then generically, they all must be equal.  However, non-generic example where this fails will arise in practice.  For instance, in Figures \ref{figAB2} and \ref{figBC2} we have $c_{3.5}=c_{4.5}$ but $b_3$ and $b_5$ are different.  In these situations, \DecorateVertexB\ will simply return a random decoration, i.e., a random line passing through $C_i$.

The subroutines \ConstructSideB\ and \DecorateSideB\ behave like \ConstructVertexB\ and \DecorateVertexB\ respectively.  The difference is that they operate with configurations which are projectively dual to the ones in Figures \ref{figConVert2} and \ref{figDecVert}.  First, Figure \ref{figConSide2} shows the triangle used by \ConstructSideB.  This procedure constructs $c_j$ in the dual manner to how \ConstructVertexB\ finds $C_i$ in Figure \ref{figConVert2}.  

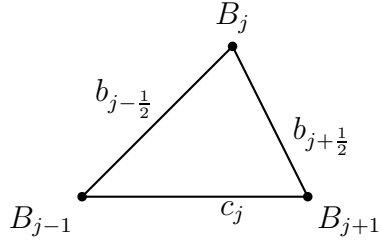
\begin{figure}
\begin{pspicture}(5,4)
\pspolygon[showpoints=true](1,1)(3,3)(4,1)
\uput[dl](1,1){$B_{j-1}$}
\uput[u](3,3){$B_j$}
\uput[dr](4,1){$B_{j+1}$}

\uput{2pt}[ur](3.75,1.5){$b_{j+\frac{1}{2}}$}
\uput{2pt}[d](3,1){$c_j$}
\uput{2pt}[ul](2,2){$b_{j-\frac{1}{2}}$}
\end{pspicture}
\caption{The triangle used in \ConstructSideB}
\label{figConSide2}
\end{figure}

Similarly, Figure \ref{figDecSide} contains a configuration that is projectively dual to the complete quadrilateral in Figure \ref{figDecVert}.  As such, \DecorateSideB\ can find $c_j^*$ via the projective dual of the construction used by \DecorateVertexB.

The case where \DecorateVertexB\ outputs a random line, and the analogous case of \DecorateSideB, are not currently justified.  However, all other cases are covered by Propositions \ref{propThirdVert} and \ref{propMultiConj}.  Hence we have the following correctness property of Algorithm \ref{algTLift2}.

\begin{prop}
\label{propT2Works}
Let $A(t)$ be a curve in the space of twisted polygons that is generic away from $t=0$.  Let $B(t) = T(A(t))$ and $C(t)=T(B(t))$ for $t \neq 0$.  Let $\mathcal{A}$, $\mathcal{B}$, and $\mathcal{C}$ be the decorated polygons associated to these curves.  Suppose that
\begin{align}
\label{condition1}
(B_{j-1} = B_{j+1} \Longrightarrow A_{j-\frac{1}{2}} = A_{j+\frac{1}{2}} = B_{j-1}) &\textrm{ for all $j$} \\
\label{condition2}
(c_{i-\frac{1}{2}} = c_{i+\frac{1}{2}} \Longrightarrow b_{i-1}=b_{i+1} = c_{i-\frac{1}{2}}) &\textrm{ for all $i$}
\end{align}
where $j$ and $i$ run over the vertex indices of $B$ and $C$ respectively.  Then $\mathcal{C} = \tilde{T}_2(\mathcal{A},\mathcal{B})$.
\end{prop}
More specifically, if \eqref{condition1} fails for some $j$, then \DecorateSideB\ chooses a random decoration for $c_j$.  If \eqref{condition2} fails for some $i$, then \DecorateVertexB\ chooses a random decoration for $C_i$.  Otherwise, the algorithm $\tilde{T}_2$ behaves deterministically and correctly. 

\section{The main algorithm}
The goal of our main algorithm is to construct $B = T^m(A)$ from $A$ when the usual construction fails, i.e. when $A$ is a singular point of various $T^k$ for $k<m$.  According to the previous section, it is typically possible to construct and decorate $T^2(A)$ given $A$, $T(A)$, and the corresponding decorations, even when singularities arise.  The main construction, given in Algorithm \ref{algMain}, simply iterates this procedure.

\begin{algorithm}
\caption{\main($A$, $m$)}
\begin{algorithmic}
\STATE $\mathcal{A}$ := \DecorateRandomly($A$)
\STATE Iterates[0] := $\mathcal{A}$
\STATE Iterates[1] := $\tilde{T}(A)$
\FOR{$k:=2$ \TO $m$}
\STATE Iterates[$k$] := $\tilde{T}_2$(Iterates[$k-2$], Iterates[$k-1$])
\ENDFOR
\STATE $\mathcal{B}$ := Iterates[$m$]
\RETURN $B$
\end{algorithmic}
\label{algMain}
\end{algorithm}

Given $S \subseteq \{1,2,\ldots,n\}$ such that singularity confinement holds on $X_S$, let $m$ be the smallest positive integer such that $T^m$ is generically defined on $X_S$.  We want to say for generic $A \in X_S$ that the main algorithm, given $A$ and $m$ as input, produces $T^m(A)$.  For the simplest singularity types, $S = \{i\}$, this result follows from Propositions \ref{propTTildeWorks} and \ref{propT2Works}.  

For more complicated $S$, a difficulty arises because the assumptions \eqref{condition1} and \eqref{condition2} in Proposition \ref{propT2Works} will not hold at every step.  Hence, some applications of $\tilde{T}_2$ in the main algorithm will produce random decorations.  To prove correctness of the algorithm for such $S$, it is necessary to determine at which steps this occurs and to demonstrate that the outcome is independent of the random choices.  

We will focus our attention on the types covered by Theorem \ref{thmSingConf}.  Taking $i=0$ for convenience in the Theorem, let $S=[-(m-1),(m-1)]_2$.  Suppose $A \in X_S$ is generic.  Tracing through the beginning of the main algorithm, let $\mathcal{A}$ be some decoration of $A$, let $\mathcal{B} = \tilde{T}(\mathcal{A})$, and let $\mathcal{C} = \tilde{T}_2(\mathcal{A},\mathcal{B})$.  

Since $A \in X_S$, the $A_i$ for $i \in [-m-1,m+1]_2$ all lie on a common line, say $l$.  It follows (see e.g. Figures \ref{figAB1}, \ref{figBC1}, \ref{figAB2}, and \ref{figBC2}) that $b_i=l$ for $i \in [-m,m]_2$ and $c_j = l$ for $j \in [-m+1/2,m-1/2]_1$.  Consequently, condition \eqref{condition2} holds for $i \in [-m+1,m-1]_2$ but fails for $i \in [-m+2,m-2]_2$ (assuming $m \geq 2$).  As such the corresponding $C_i$ are decorated randomly.

To establish that \main($A$, $m+2$) = $T^{m+2}(A)$ for $A$ as above, we need to prove two facts.  The first is that the output of the algorithm does not depend on the decorations of the $C_i$ that are chosen randomly.  For given $m$, we can check this computationally by showing that any such choice of decorations is possible for an appropriate choice of representative $A(t)$ of $\mathcal{A}$.  The second fact is that no other violations of \eqref{condition1} or \eqref{condition2} occur until computing Iterates[$m+2$] in the last step.  For given $m$, it suffices to check that this fact is true for a single $\mathcal{A}$ as it then follows for generic examples.

We have no general proof for the necessary facts, but we have verified that they hold for the first several values of $m$.  Assuming them, we can repeatedly apply Proposition \ref{propT2Works} to conclude that Iterates[$k$] is the decoration corresponding to the curve $T^k(A(t))$ for some curve $A(t)$ through $A$ and all $k < m+2$.  At the last step condition \eqref{condition1} will fail in some places, so Iterates[$k+2$] will have some randomly decorated sides.  However, the sides of Iterates[$k+2$] themselves will be correct proving that its underlying polygon is in fact $T^{m+2}(A)$.    Hence we get that the main algorithm works correctly for polygons of type $[-m+1,m-1]_2$ when $m$ is small.  We expect that this result holds for all $m$.

\begin{rem}
The main algorithm is stated without regard to a particular singularity type.  Hence it has the potential to work in greater generality than is discussed above.  Experiments indicate that the algorithm does work for many, but not all, other singularity types.  The simplest types for which it fails are $S=\{3,4,6\}$ and similar.
\end{rem}

\appendix
\section{Some basic constructions}
This appendix states and proves straightedge constructions for the primitives used in the algorithms throughout the paper.  The first, namely \TripleConjugate, is given in Algorithm \ref{algTripConj}.  This construction was shown to me by Pavlo Pylyavskyy.

\begin{algorithm}
\caption{\TripleConjugate($A,B,C,D,E$)}
\begin{algorithmic}
\STATE $P$ := \RandomPoint()
\STATE $C'$ := \RandomPointOn($\join{C}{P}$)
\STATE $B'$ := $\meet{\join{B}{P}}{\join{A}{C'}}$
\STATE $D'$ := $\meet{\join{D}{P}}{\join{C'}{E}}$
\STATE $F$ := $\meet{\join{B'}{D'}}{\join{A}{E}}$
\RETURN $F$
\end{algorithmic}
\label{algTripConj}
\end{algorithm}

This algorithm constructs points $B'$, $C'$, $D'$ and $F$ such that $(A,B',C',D',E,F)$ is a Menelaus configuration (see Figure \ref{figTripConj}).  Therefore $[A,B',C',D',E,F]=-1$.  Applying a projective transformation we may assume that $P$ is a point at infinity.  The lines $\join{B}{B'}$, $\join{C}{C'}$, and $\join{D}{D'}$ all pass through this point, so they must be parallel.  Therefore $\triangle ABB'$ is similar to $\triangle ACC'$, so $\frac{AB}{BC} = \frac{AB'}{B'C'}$.  Also, $\triangle EDD'$ is similar to $\triangle ECC'$, so $\frac{CD}{DE} = \frac{C'D'}{D'E}$.  It follows that
\begin{displaymath}
[A,B,C,D,E,F] = [A,B',C',D',E,F] = -1
\end{displaymath}
as desired.

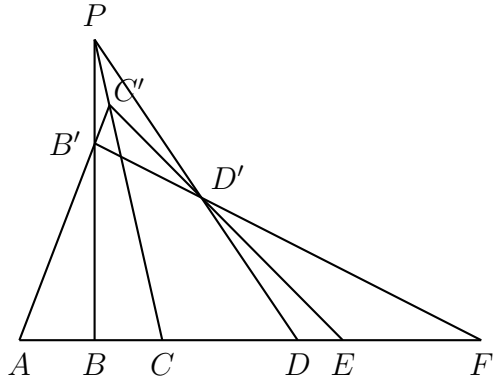
\begin{figure}
\begin{pspicture}(8,6)
\pnode(1,1){A}
\pnode(2,1){B}
\pnode(2.9,1){C}
\pnode(4.7,1){D}
\pnode(5.3,1){E}
\pnode(7.14,1){F}
\pnode(2,5){P}
\pnode(2,3.62){B1}
\pnode(2.2,4.13){C1}
\pnode(3.42,2.89){D1}
\uput[d](A){$A$}
\uput[d](B){$B$}
\uput[d](C){$C$}
\uput[d](D){$D$}
\uput[d](E){$E$}
\uput[d](F){$F$}
\uput[u](P){$P$}
\uput[l](B1){$B'$}
\uput{2pt}[ur](C1){$C'$}
\uput[ur](D1){$D'$}

\ncline{A}{F}
\ncline{C}{P}
\ncline{A}{C1}
\ncline{E}{C1}
\ncline{B}{P}
\ncline{D}{P}
\ncline{B1}{F}

\end{pspicture}
\caption{The construction of \TripleConjugate}
\label{figTripConj}
\end{figure}

Next, Algorithm \ref{algProjTrans} inputs four points $A$, $B$, $C$, $D$, on one line, and three points $A'$, $B'$, $C'$ on another.  There exists a unique projective transformation from the first line to the second taking $A$ to $A'$, $B$ to $B'$, and $C$ to $C'$.  The algorithm returns the result of applying this projective transformation to $D$.

\begin{algorithm}
\caption{\ProjectiveTransformation($A,B,C,D,A',B',C'$)}
\begin{algorithmic}
\STATE $l$ := \RandomLineThrough($A'$)
\STATE $P$ := \RandomPointOn($\join{A}{A'}$)
\STATE $B''$ := $\meet{l}{\join{B}{P}}$
\STATE $C''$ := $\meet{l}{\join{C}{P}}$
\STATE $Q$ := $\meet{\join{B''}{B'}}{\join{C''}{C'}}$
\STATE $D''$ := $\meet{l}{\join{D}{P}}$
\STATE $D'$ := $\meet{\join{D''}{Q}}{\join{A'}{B'}}$
\RETURN $D'$
\end{algorithmic}
\label{algProjTrans}
\end{algorithm}

The algorithm selects a line $l$ and a point $P$ such that projection through $P$ onto $l$ sends $A$ to $A'$.  The images of $B$, $C$, are called $B''$ and $C''$ respectively.  Then $Q$ is constructed so that projection through $Q$ onto the target line sends $B''$ to $B'$ and $C''$ to $C'$, while necessarily fixing $A'$ (see Figure \ref{figProjTrans}).  Hence, the composition of these two projections is the desired projective transformation.  Applying it to $D$ gives the output $D'$.

\begin{figure}
\begin{pspicture}(6,5)
\cnode*(1,1){1.8pt}{A}
\cnode*(2.8,1){1.8pt}{B}
\cnode*(5.6,1){1.8pt}{C}
\cnode*(4.46,4.63){1.8pt}{A1}
\cnode*(2.47,4.44){1.8pt}{B1}
\cnode*(.84,4.28){1.8pt}{C1}
\cnode*(1.75,2.8){1.8pt}{B2}
\cnode*(1.25,2.46){1.8pt}{C2}
\cnode*(2.12,2.17){1.8pt}{P}
\cnode*(1.37,1.93){1.8pt}{Q}

\uput[d](A){$A$}
\uput[d](B){$B$}
\uput[d](C){$C$}
\uput[u](A1){$A'$}
\uput[u](B1){$B'$}
\uput[u](C1){$C'$}
\uput{2pt}[ul](B2){$B''$}
\uput[l](C2){$C''$}
\uput[20](P){$P$}
\uput[l](Q){$Q$}

\ncline{A}{C}
\ncline{A1}{C1}
\ncline{A}{A1}
\ncline{A1}{C2}
\nbput[labelsep=2pt]{$l$}
\ncline{B}{B2}
\ncline{C}{C2}
\ncline{B1}{Q}
\ncline{C1}{Q}

\end{pspicture}
\caption{Part of the construction used in \ProjectiveTransformation}
\label{figProjTrans}
\end{figure}
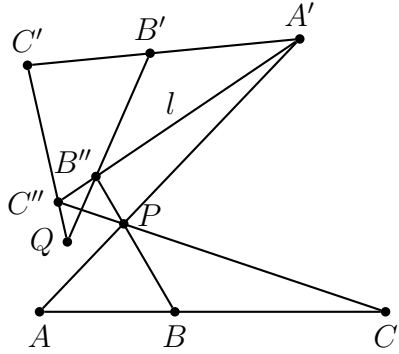

Cross ratios are invariant under projective transformation.  Hence $[A',B',C',D'] = [A,B,C,D]$ and $D'$ is the unique point on the line containing $A',B',C'$ with this property.  As such, we use this construction to find a point $B$ satisfying \eqref{degTriangle}.  This appears to be a more complicated situation because one of the cross ratios is inverted, and also because both points and lines are involved.  The identity can be expressed in terms of points alone using the fact that $[l_1,l_2,l_3,l_4] = [\meet{l_1}{l},\meet{l_2}{l},\meet{l_3}{l},\meet{l_4}{l}]$ for any other line $l$.  The reciprocal can be eliminated by reordering via the property $[P_1,P_2,P_3,P_4]^{-1} = [P_1,P_4,P_3,P_2]$.  

Another component of several of our algorithms involves finding one point or line from \eqref{triangle} in terms of the others.  By similar remarks to before, it is possible to cast this as a problem involving points alone, namely to construct $F'$ from the other points assuming
\begin{displaymath}
[A,B,C,D,E,F] = [A',B',C',D',E',F']
\end{displaymath}
Here, not all points are assumed to be collinear, only those triples required by the definition of triple ratios.

\begin{lem}
$[A,B,C,D,E,F] = [A,P,E,F]$ where
\begin{displaymath}
P = \meet{\join{C}{(\meet{\join{A}{D}}{\join{B}{E}})}}{\join{A}{E}}
\end{displaymath}
(see Figure \ref{figCeva}).
\end{lem}
\begin{proof}
For any point $P$ on $\meet{A}{E}$, we have
\begin{displaymath}
\frac{[A,B,C,D,E,F]}{[A,B,C,D,E,P]} = [A,P,E,F]
\end{displaymath}
For the particular $P$ chosen, Ceva's theorem guarantees that $[A,B,C,D,E,P] = 1$.
\end{proof}

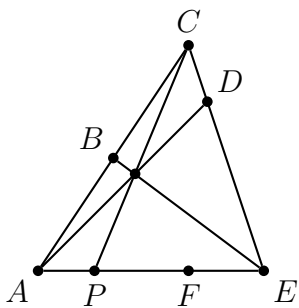
\begin{figure}
\begin{pspicture}(5,5)
\cnode*(1,1){2pt}{A}
\cnode*(2,2.5){2pt}{B}
\cnode*(3,4){2pt}{C}
\cnode*(3.25,3.25){2pt}{D}
\cnode*(4,1){2pt}{E}
\cnode*(3,1){2pt}{F}
\cnode*(2.29,2.29){2pt}{Q}
\cnode*(1.75,1){2pt}{P}

\uput[dl](A){$A$}
\uput[ul](B){$B$}
\uput[u](C){$C$}
\uput[ur](D){$D$}
\uput[dr](E){$E$}
\uput[d](F){$F$}
\uput[d](P){$P$}

\ncline{A}{C}
\ncline{C}{E}
\ncline{E}{A}
\ncline{A}{D}
\ncline{C}{P}
\ncline{E}{B}
\end{pspicture}
\caption{The construction of a point $P$ satisfying $[A,B,C,D,E,F] = [A,P,E,F]$}
\label{figCeva}
\end{figure}

In light of this lemma, it is easy to construct the point $F'$ above.

\begin{algorithm}
\caption{\ProjectiveTransformationB($A,B,C,D,E,F,A',B',C',D',E'$)}
\begin{algorithmic}
\STATE $P$ := $\meet{\join{C}{(\meet{\join{A}{D}}{\join{B}{E}})}}{\join{A}{E}}$
\STATE $P'$ := $\meet{\join{C'}{(\meet{\join{A'}{D'}}{\join{B'}{E'}})}}{\join{A'}{E'}}$
\STATE $F'$ := \ProjectiveTransformation($A,P,E,F,A',P',E'$)
\RETURN $F'$
\end{algorithmic}
\label{algProjTrans2}
\end{algorithm}

\end{document}